\documentclass{svjour3}                     % onecolumn (standard format)
\RequirePackage{fix-cm}
\smartqed  % flush right qed marks, e.g. at end of proof
\usepackage[utf8]{inputenc}
\usepackage[T1]{fontenc}
\usepackage{amssymb}
\usepackage{amsfonts}
\usepackage{amsmath}
\usepackage{graphicx}
\usepackage{subfig} 
\usepackage{textcomp}
\usepackage{mathtools}
\usepackage{fullpage}
\usepackage{enumerate}
\usepackage{bm}
\numberwithin{equation}{section}
\newcommand{\vertiii}[1]{{\left\vert\kern-0.25ex\left\vert\kern-0.25ex\left\vert #1 
		\right\vert\kern-0.25ex\right\vert\kern-0.25ex\right\vert}}

\newcommand{\Bu}{\mathbf{u}}
\newcommand{\R}{\mathbb{R}}
\newcommand{\C}{\mathbb{C}}
\newcommand{\CE}{\mathcal{E}}
\newcommand{\CX}{\mathcal{X}}
\newcommand{\CN}{\mathcal{N}}
\newcommand{\CH}{\mathcal{H}}
\newcommand{\BE}{\mathbf{E}}
\newcommand{\BA}{\mathbf{A}}
\newcommand{\BB}{\mathbf{B}}
\newcommand{\BD}{\mathbf{D}}
\newcommand{\BY}{\mathbf{Y}}
\newcommand{\BU}{\mathbf{U}}
\newcommand{\BV}{\mathbf{V}}
\newcommand{\BI}{\mathbf{I}}
\newcommand{\BL}{\mathbf{L}}
\newcommand{\BLT}{\mathbf{L}^{\mkern-2.7mu 2}}

\newcommand{\BS}{\mathbf{S}}
\newcommand{\BX}{\mathbf{X}}
\newcommand{\Bx}{\mathbf{x}}
\newcommand{\BH}{\mathbf{H}}
\newcommand{\BF}{\mathbf{F}}
\newcommand{\Bn}{\mathbf{n}}
\newcommand{\Bz}{\mathbf{z}}

\newcommand{\Bp}{\mathbf{p}}

\newcommand{\Bv}{\mathbf{v}}

\newcommand{\BPi}{\mathbf{\Pi}}
\newcommand{\Br}{\mathfrak{r}}
\newcommand{\Bw}{\mathfrak{w}}

\newcommand{\diam}{\operatorname*{diam}}

\newcommand{\abs}[1]{\left\vert #1 \right\vert}
\newcommand{\skp}[1]{\left< #1 \right>}
\newcommand{\norm}[1]{\left\| #1 \right\|}

\newcommand{\T}{\mathcal{T}}
\newcommand{\TT}{{\mathcal{T}}}
\newcommand{\eremk}{\hbox{}\hfill\rule{0.8ex}{0.8ex}}

\newtheorem{assumption}[theorem]{Assumption}

\begin{document}

\title{$\mathcal{H}$-matrix approximability of inverses of FEM matrices for the time-harmonic Maxwell equations
%\thanks{Financial support by the Austrian Science Fund (FWF) through the
%research program ``Taming complexity in partial differential systems'' (grant SFB F65) for JMM, through grant P 28367-N35 for JMM and MP  \mp{and by the Deutsche Forschungsgemeinschaft (DFG) under Germany’s
%Excellence Strategy within the Cluster of Excellence PhoenixD (EXC 2122, 
%	Project ID 390833453) for MP} is gratefully acknowledged.}
}


\author{Markus Faustmann         \and
        Jens Markus Melenk   \and
        Maryam Parvizi
}


\institute{Markus Faustmann \at
              TU Wien, Institute of Analysis and Scientific Computing, Wiedner Hauptstra\ss{}e 8-10, 1040 Wien, Austria \\
              \email{markus.faustmann@tuwien.ac.at}           %  \\
           \and
	Jens Markus Melenk \at
              TU Wien, Institute of Analysis and Scientific Computing, Wiedner Hauptstra\ss{}e 8-10, 1040 Wien, Austria \\
              \email{melenk@tuwien.ac.at}
           \and
	Maryam Parvizi \at
              Leibniz University of Hannover, Institute of Applied Mathematics,  Welfengarten 1, D--30167  Hannover, Germany \\
              Cluster of Excellence PhoenixD (Photonics, Optics, and Engineering - 
              	Innovation Across
              	Disciplines), Leibniz Universität Hannover, Germany\\
              \email{parvizi@ifam.uni-hannover.de}
}


\maketitle

\begin{abstract}
The inverse of the stiffness matrix of the time-harmonic Maxwell equation with perfectly conducting boundary conditions is approximated in the 
blockwise low-rank format of ${\mathcal H}$-matrices.  Under a technical assumption on the mesh, we prove that root exponential convergence in the block rank can be achieved if the 
block structure conforms to a standard admissibility criterion. 
\end{abstract}

\section{Introduction}
A backbone of computational electromagnetics is the solution of the time-harmonic Maxwell equations. Since the discovery of N\'ed\'elec's edge elements (and their 
higher order generalizations) finite element methods (FEMs) have become an important discretization technique for these equations with an established
convergence theory, \cite{monk2003finite}. While the resulting linear system is sparse, a direct solver cannot achieve linear complexity as one has to expect,
already for the case of quasiuniform meshes with problem size $N$, a complexity $O(N^{4/3})$ for the memory requirement and $O(N^2)$ for the solution time of 
a multifrontal solver, \cite{liu1992multifrontal}. Iterative solvers such as multigrid or preconditioned Schwarz methods can lead 
to optimal (or near optimal) complexity for the numerical solution of the time-harmonic Maxwell equations, at least in the low-frequency regime, 
\cite{hiptmair99,arnold-falk-winther00,gopalakrishnan2003overlapping}.  For the design and analysis 
of these methods, a key insight was the appropriate treatment of the gradient part of the N\'ed\'elec space and thus Helmholtz decompositions play 
an important role. The analysis of fast solvers for Maxwell's equations, however, is less developed in areas such as high-frequency applications.  

An alternative to classical direct solvers and iterative solvers came with the introduction of $\CH$-matrices in 
\cite{Hackbusch99}. This class of matrices consists of blockwise low-rank matrices of rank $r$, where the blocks are organized in a tree ${\mathbb T}_\mathcal{I}$ 
so that the memory requirement is typically $O(r N\operatorname{depth}({\mathbb T}_\mathcal{I}))$, where $N$ is the problem size. This format comes with an (approximate) 
arithmetic that allows for addition, multiplication, inversion, and $LU$-factorization in logarithmic-linear complexity. Therefore, computing 
an (approximate) inverse in the $\CH$-format can be considered a serious alternative to a direct solver or it can be used as a ``black box'' preconditioner 
in iterative solvers. We refer to the works \cite{hackbusch2015hierarchical,grasedyck01,grasedyck-hackbusch03,borm2010efficient} for a more detailed discussion of 
analytical and algorithmic aspects of $\CH$-matrices. 

A basic question in connection with the $\CH$-matrix arithmetic is whether matrices and their inverses or factors in an $LU$-factorization can
be represented well in the chosen format. While stiffness matrices arising from differential operators are sparse and are thus easily represented
exactly in the standard $\CH$-matrix formats, the situation is more involved for the inverse. A first proof that inverses can be represented in the 
${\mathcal H}$-matrix format harks back to \cite{BebendorfHackbusch,Bebendorf07} for scalar elliptic problems and \cite{bebendorf2009parallel} for 
the $\operatorname{curl} \mu^{-1} \operatorname{curl}$ operator; a generalization to pseudodifferential operators is done in 
\cite{doelz-harbrecht-schwab17}. These proofs rely on locally separabel approximations of the continuous Green's function and a final projection 
of these approximations into discrete spaces. The final projection step limits, at least formally, the achievable accuracy of the matrix approximation 
by the discretization error. To circumvent this,  a fully discrete approach was taken for FEM discretizations of various scalar elliptic operators 
in \cite{faustmann2015mathcal,angleitner-faustmann-melenk20} to show  that the inverse of the FEM-matrix can be approximated
at a root exponential rate in the block rank.
The works \cite{FMP16,FMP17} extend these results to the boundary element method (BEM) and 
\cite{faustmann-melenk-parvizi19} to a FEM-BEM coupling setting. The underlying mechanism in these works is that ellipticity of the operator allows 
one to prove a discrete Caccioppoli inequality, where a higher order norm (e.g., the $H^1$-norm) is controlled by a lower order norm (e.g., the $L^2$-norm)
on a slightly larger region. This gain in regularity can be exploited for approximation purposes, and an exponential approximation can be obtained 
by iterating the argument.  The present setting of Maxwell's equations is different since 
the corresponding Caccioppoli inequality (Lemma~\ref{th:Caccioppoli-divergence free part}) controls only the $\BH(\operatorname{curl})$-norm 
by the $\BLT$-norm. Since $\BH(\operatorname{curl})$ is not compactly embedded in $\BLT$, this Caccioppoli inequality is insufficient for approximation
purposes. We therefore combine this Caccioppoli inequality with a local discrete Helmholtz-type decomposition. The gradient part can be treated with 
techniques established in \cite{faustmann2015mathcal} for Poisson problems, whereas the remaining part can, up to a small perturbation,  be controlled 
in $\BH^1$ so that approximation becomes feasible and one may proceed structurally similarly to the scalar case. The local discrete 
Helmholtz-type decomposition (Lemma~\ref{local-stability-approx-divergence-free}) may also be of independent interest.

This paper is organized as follows. In Section~\ref{sec:mainresults}, we introduce  the time-harmonic Maxwell equations and their discretization with 
N\'ed\'elec's $\operatorname*{curl} $-conforming elements. We state the main result of this paper, namely, the existence of 
$\mathcal{H}$-matrix approximations to the inverse stiffness matrix that converge root exponentially in the block rank. We hasten to add that we do not 
track the dependence on the frequency $\omega$ in our analysis and focus on the case of fixed wave number $\kappa$. As in the case of the Helmholtz equation, the high-frequency case of $\omega\rightarrow \infty$ would require specialized matrix formats such as directional $\mathcal{H}^2$-matrices (${\mathcal D}{\mathcal H}^2$) or the butterfly format; we refer to the 
literature discussions in \cite{bebendorf-kuske-venn15,boerm-melenk17,boerm-boerst-melenk17}. 
To prove the approximability result of Section~\ref{sec:mainresults}, we present 
in Section~\ref{sec:helmholtz-decomposition} a local discrete Helmholtz decomposition
and prove stability and approximation properties of this decomposition under a certain technical assumption on the mesh. In Section~\ref{sec:low-dimensional}, we present a Caccioppoli-type inequality
for discrete $\boldsymbol{\mathcal L}$-harmonic functions with $\boldsymbol{\mathcal{L}}$ being the Maxwell operator. 
Furthermore, we obtain exponentially convergent approximations to discrete $\boldsymbol{\mathcal L}$-harmonic functions. 
Section~\ref{sec:proof} is concerned  with  the proof of the main result of this paper.

\medskip 
Concerning notation: Constants $C>0$ may differ in different occurrences but are independent of critical parameters such as the mesh size. 
$a \lesssim b$ indicates the existence of a constant $C >0$ such that $a \leq C b$. For a set $A \subset {\mathbb R}^3$, we denote by $|A|$ 
its Lebesgue measure. For finite sets $B$, the cardinality of $B$ is also denoted by $|B|$. 
We employ standard Sobolev spaces as described in \cite{mclean00}. 
We also denote $\Omega^c:= {\mathbb R}^3 \setminus \overline{\Omega}$. 
 
 %________________________________________________________________________________________________________
 \section{Main results}\label{sec:mainresults}
%_____________________________________________________________________________________________________________--
\subsection{Model problem}
Maxwell's equations are a system of first-order partial differential equations that 
connect the temporal and spatial rates of change of the electric and magnetic fields possibly in the presence of additional
source terms. Let $\Omega \subset \mathbb{R}^3$ be a  simply connected 
polyhedral   domain  with boundary $\Gamma := \partial \Omega$ that, in physical terms, is  filled with
a homogeneous isotropic material.  Maxwell's equations then connect the electric field $\CE$ to the magnetic field $\CH$ by
\begin{subequations}
	\begin{alignat}{5}
	\label{time dependent Ampere's law}
	\left(\varepsilon \frac{\partial }{\partial t }+\sigma \right) \CE- \nabla \times \CH &= \mathcal{G}\quad \hspace{1cm}&&\text{in} \,\,\Omega,\\
			\label{time dependent Faraday's law}
	\mu \frac{\partial }{\partial t }\CH+ \nabla \times \CE &=0 \quad \hspace{1cm}&&\text{in} \,\,\Omega,
	\end{alignat}  
\end{subequations}
where $\mathcal{G}$ is a given function representing the applied current. Homogeneous isotropic materials 
 can be characterized by a positive dielectric constant $\varepsilon > 0$, a positive permeability constant $\mu > 0$, 
and a non-negative electric conductivity  constant $\sigma \geq 0$. 
  In this paper, we consider perfectly conducting boundary conditions for $\CE$, i.e.,
$$
\Bn \times \CE =0 \qquad \text{on}\quad \Gamma,
$$ 
where $\Bn$ is the unit outward normal vector on $\Gamma$. 

We assume the arising fields to be time-harmonic, i.e., 
\begin{alignat}{3}
	\label{2a}
	\CE(x,t)&= e ^{-i \omega t} \BE(x), & \qquad 
\CH(x,t)= e ^{-i \omega t} \BH(x), \qquad 
\mathcal{G}(x,t)& = e ^{-i \omega t} \mathbf{J}(x)
	\end{alignat}  
for some given frequency $\omega$. 
Substituting \eqref{2a} into \eqref{time dependent Ampere's law} and \eqref{time dependent Faraday's law}, we get
\begin{subequations}
	\label{Maxwell system}
	\begin{alignat}{5}
	\label{Faraday's  law}
- \nabla \times \BH-i  \omega\eta \BE &= \mathbf{J}\quad \hspace{1cm}&&\text{in} \,\,\Omega,\\
	\label{Ampere law}
	\nabla \times \BE-i \omega \mu \BH &= 0\quad \hspace{1cm}&&\text{in} \,\,\Omega,
	\end{alignat}  
\end{subequations}
where $\eta:=  \varepsilon + {i \sigma }/{\omega}$. Finally, the first order  system \eqref{Maxwell system} can be reduced to a second order equation
\begin{align} \label{Maxwell-The electric field}
{\boldsymbol {\mathcal L}} \BE := \nabla  \times (\mu ^{-1}\nabla  \times \BE)- \kappa \BE= \BF \qquad \text{in} \,\,\Omega ,
\end{align}
where  $\kappa:=\omega ^2 \eta $ and $\mathbf{F}:=- i \omega \mathbf{J}$.
For the sake of simplicity, we also assume $\mu=1$ in the following. 

With $\BLT(\Omega):= L^2(\Omega)^3$, we define the space 
$\displaystyle \BH(\operatorname*{curl},\Omega):= \left\lbrace \mathbf{U} \in \BLT (\Omega ) \; \colon \; \nabla \times  \mathbf{U}\in \BLT (\Omega )  \right\rbrace , $
equipped with the norm 
$$\norm{\textbf{U}}^2_ {\BH (\operatorname*{curl},\Omega)}:= \norm{\textbf{U}}^2_{\BLT(\Omega)}+ \norm{\nabla \times \textbf{U}}^2_{\BLT(\Omega)},$$
and the subspace $\BH_0 (\operatorname*{curl},\Omega) \subset \BH(\operatorname*{curl},\Omega)$ with zero boundary conditions 
$$\BH_0 (\operatorname*{curl},\Omega):=\lbrace \textbf{U} \in \BLT (\Omega ) \; \colon \; \nabla \times  \textbf{U}\in \BLT (\Omega ) , \;  \Bn \times  \BU =0 \;  \text{on}  \; \Gamma \rbrace. $$

The following lemma asserts that the tangential trace operator for functions in  $\BH(\operatorname*{curl},\Omega)$  is indeed well-defined:

\begin{lemma}\cite[Thm.~{3.29}]{monk2003finite}
Let $\Omega$ be a bounded Lipschitz domain. Then, the \emph{trace operator} 
$$\gamma _T : \mathbf{C}^\infty(\overline{\Omega})\rightarrow {\mathbf{C}}^0 (\Gamma),\quad \BU \mapsto \Bn \times  \BU|_{\Gamma}$$
can be uniquely extended to a bounded linear operator $\gamma _ T: \BH (\operatorname*{curl},\Omega)\rightarrow \BH ^{-1/2} (\Gamma) $. 
\end{lemma} 
Multiplying both sides of \eqref{Maxwell-The electric field} with $\Psi \in \BH_0 (\operatorname*{curl},\Omega)$ and integrating by parts,  we obtain the weak formulation: Find  $\BE \in \BH_0 (\operatorname*{curl},\Omega)$  such that
\begin{align}
\label{Maxwell weak form}
a(\BE, \Psi):= \skp{\nabla \times \BE, \nabla \times \Psi}_{\BLT(\Omega)}  - \kappa \skp{\BE, \Psi}_{\BLT(\Omega)}=
\skp{\BF, \Psi}_{\BLT(\Omega)}\qquad \forall \Psi \in \BH_0 (\operatorname*{curl},\Omega),
\end{align}
where $\skp{\cdot, \cdot}_{\BLT(\Omega)}$ is the $\BLT (\Omega )$-inner product. 
We assume that $\kappa$ is not an eigenvalue of the  operator $\nabla \times \nabla \times$, see, e.g., \cite[Sec.~{4}]{monk2003finite}.  
This implies in particular that $\kappa \ne 0$ since $\nabla H^1_0(\Omega)$ is contained in the kernel of the operator $\nabla \times \nabla \times $.  
Then, the Fredholm
alternative  provides the existence of  a unique  solution  to the  variational problem, 
and we have  the {\sl a priori} estimate
\begin{equation}
\label{eq:apriori}
\norm{\BE}_{\BH(\operatorname{curl},\Omega)} \leq C_{\rm stab} \norm{\BF}_{\BLT(\Omega)}
\end{equation}
for a constant $C_{\rm stab}$ that depends on $\Omega$ and $\kappa$, see, e.g., \cite[Thm.~5.2]{hiptmair2002finite}.
%
%--------------------------------------------------------------------------------------------------
%
\subsection{Discretization by edge elements}
Let ${\mathcal T}_h=\{T_1,\dots,T_{N_{\mathcal{T}}}\}$ be  a  quasi-uniform triangulation of 
 $\Omega$  with the  mesh width 
$h := \max_{T_j\in \mathcal{T}_h}{\rm diam}(T_j)$, where 
the elements $T_j \in \mathcal{T}_h$ are open  tetrahedra. 
The mesh $\T_h$ is assumed to be  regular in the  sense of Ciarlet, i.e., there are no hanging nodes. 
The assumption of quasi-uniformity includes the assumption of $\gamma$-shape regularity, i.e., 
there is $\gamma >0$ such that ${\rm diam}(T_j) \le \gamma\,|T_j|^{1/3}$ for all $T_j\in\mathcal{T}_h$.
 For the Galerkin discretization of  \eqref{Maxwell weak form}, 
we use lowest order N\'ed\'elec's $\BH (\operatorname*{curl} , \Omega)$-conforming elements of the first kind, see, e.g., \cite[Sec.~5]{monk2003finite}. 
That is, on  $T \in \T _h $, we introduce the lowest order local N\'ed\'elec space
$$
\CN _0 (T) = \{\bold a  + \bold b \times \Bx \; \colon \; \bold a , \bold b \in \R ^3, \quad \Bx \in T\},
$$
and set 
\begin{align*}
\BX_h(\T _h,\Omega) &:=\lbrace \BU _h \in \BH(\operatorname*{curl},\Omega)\; \colon \; \BU _h |_T \in \CN _0 (T)   \quad \forall T \in \mathcal{T}_h \rbrace, \\
\BX_{h,0}(\T _h,\Omega) &:=\BX_h(\T_h,\Omega) \cap \BH_0(\operatorname{curl},\Omega). 
\end{align*}
The standard degrees of freedom of $\BX_h(\T_h,\Omega)$ are the line integrals of the tangential component of $\BU_h$ on the edges of $\T_h$, 
see, e.g., \cite[Sec.~{5.5.1}]{monk2003finite}, \cite[Sec.~{2.3.2}]{boffi-brezzi-fortin13}. Hence, the dimension of $\BX_h(\T_h,\Omega)$ is the number of 
edges of $\T_h$. The standard basis $\CX_h:= \{\Psi_e\}$ of $\BX_h(\T_h,\Omega)$ consists of the so-called (lowest order) edge elements, 
where the function $\Psi_e \in \BX_h(\T_h,\Omega)$ is associated with the edge $e$ of $\T_h$ and is supported by the union of the tetrahedra sharing the edge $e$. 
More specifically, for an edge $e$ with endpoints $V_1$, $V_2$ and a tetrahedron $T$ with edge $e$, one has 
$\Psi_e|_T =  \lambda_{V_1} \nabla \lambda_{V_2} - \lambda_{V_2} \nabla \lambda_{V_1}$, where $\lambda_{V_i}$ is the standard hat function associated
with vertex $V_i$. 

A basis $\CX_{h,0}:= \{\Psi_1,\ldots,\Psi_N\}$ of $\BX_{h,0}(\T_h,\Omega)$ with $N:= \operatorname{dim} \BX_{h,0}(\T_h,\Omega)$ 
is obtained by taking the $\Psi_e \in \CX_h$, whose edge $e$ satisfies $e \subset \Omega$; that is, $\CX_{h,0}$ is obtained from $\CX_h$ by removing 
the shape functions associated with edges lying on $\Gamma$. 

Using $\BX _{h,0}(\T _h,\Omega )\subseteq \BH_0(\operatorname*{curl},\Omega)$ as ansatz  and test space in  \eqref{Maxwell weak form}, 
we arrive at the Galerkin discretization of  finding 
$\BE _h \in \BX _{h,0}(\T _h,\Omega)$ such that 
\begin{align}
\label{Galerkin discretization}
a(\BE_h, \Psi_h)=
\skp{\BF  , \Psi _h }_{\BLT(\Omega)}\qquad \forall \Psi _h \in \BX _ {h,0}(\T _h,\Omega).
\end{align}
Using the basis $\CX_{h,0}$, the Galerkin discretization  (\ref{Galerkin discretization}) can be formulated as a linear system of equations,
where the system matrix $\BA \in \mathbb{C} ^ {N \times N} $ is given by 
\begin{align}
\label{eq:stiffness-matrix}
\BA _ {ij}:= a(\Psi _j , \Psi _i), \qquad  \Psi _j , \,\Psi _i \in \CX _ {h,0}.
\end{align}

For unique solvability of the discrete problem (\ref{Galerkin discretization}) or, equivalently, the invertibility of $\BA$, 
we recall the following Lemma~\ref{Lem. stability}. In that result and throughout the paper, we denote by 
\begin{equation}
\BPi^{L^2} _h:  \BLT(\Omega) \rightarrow \BX_h(\T_h,\Omega)
\end{equation}
the $\BLT ( \Omega)$-orthogonal projection onto $\BX_h(\T _h,\Omega)$. 
\begin{lemma}\label{Lem. stability}
	\cite[Thm.~{5.7}]{hiptmair2002finite}
Assume (\ref{eq:apriori}).
There  exists $h_0>0$
depending on the parameters of the continuous problem and the $\gamma$-shape regularity  of  $\T _h$, such that, for $h < h _0$, the discrete  problem \eqref{Galerkin discretization} has a unique solution and there holds the stability estimate
\begin{align*}
\norm{\BE _h } _ {\BH  (\operatorname*{curl},\Omega)}\le C \norm{\BPi^{L^2} _h  \BF }_{\BLT (\Omega)}.
\end{align*}
 Here, $C>0$ is a  constant depending solely on the $\gamma$-shape regularity  of  $\T _h$ and the parameters of the continuous problem.
\end{lemma}

%-----------------------------------------------------------------------------------------------------------------------
%
\subsection{Hierarchical matrices}
The goal of this paper is to obtain an $\mathcal{H}$-matrix approximation of  the inverse matrix $\BB:= \BA^{-1}$. An
$\mathcal{H}$-matrix
is a blockwise low rank matrix, where suitable  blocks  for low rank approximation are chosen by the 
  concept of \emph{admissibility}, which is defined in the following.
\begin{definition}[bounding boxes and $\eta$-admissibility] A cluster $\tau$ is a subset of the index set 
		$\mathcal{I}=\{1,2,\ldots,N\}$. 
		For a cluster $\tau \subset \mathcal{I}$, an axis-parallel box $B_{R_\tau} \subseteq \mathbb{R}^3$ is called a \emph{bounding box}, 
		if $B_{R_\tau}$ is a cube with side length $R_\tau$ and 
		$\cup _ {i \in \tau} \operatorname*{supp} \Psi _i \subseteq B_{R_\tau}$.
\\
	Let  $\eta >0$. Then, a pair of clusters $(\tau , \sigma)  \subset \mathcal{I} \times \mathcal{I} $  is called \emph{$\eta$-admissible}, if 
	there exist bounding boxes $B_{R_\tau}$ and $B_{R_\sigma}$ of $\tau$ and $\sigma$ such that 
	\begin{align}
\label{eq:admissibility-condition}
	\min \{\operatorname*{diam}(B_{R_\tau}), \operatorname*{diam}(B_{R_\sigma})\} \leq 
	\eta \, \operatorname*{dist} (B_{R_\tau},B_{R_\sigma}).
	\end{align}
\end{definition}
\begin{definition}[Concentric boxes]
Axis-parallel boxes $B_R$ of side length $R$ are called \emph{boxes}. 
	Two  boxes $B_R$ and $B_{R^\prime}$ of side length $R$ and ${R^\prime}$
	are said to be \emph{concentric}, if they have the same barycenter and $B_R$
	can be obtained by a stretching of $B_{R^\prime}$
	by the factor $R/R^\prime$ taking their common barycenter as the origin.
\end{definition}

\begin{definition}[cluster tree]
	A \emph{cluster tree} with \emph{leaf size} $n_{\rm leaf} \in \mathbb{N}$ is a binary tree 
	$\mathbb{T}_{\mathcal{I}}$ with root $\mathcal{I}$  
	such that each cluster $\tau \in \mathbb{T}_{\mathcal{I}}$ is either a leaf of the tree 
	and satisfies $\abs{\tau} \leq n_{\rm leaf}$, or there exist disjoint subsets 
	$\tau'$, $\tau'' \in \mathbb{T}_{\mathcal{I}}$ of $\tau$, called \emph{sons}, with 
	$\tau = \tau' \dot{\cup} \tau''$. 
	The \emph{level function} ${\rm level}\colon \mathbb{T}_{\mathcal{I}} \rightarrow \mathbb{N}_0$ 
	is inductively defined by 
	${\rm level}(\mathcal{I}) = 0$ and ${\rm level}(\tau') := {\rm level}(\tau) + 1$ for $\tau'$ a son of $\tau$. 
	Furthermore,
	 ${\rm depth}(\mathbb{T}_{\mathcal{I}}) := \max_{\tau \in \mathbb{T}_{\mathcal{I}}}{\rm level}(\tau)$ is called the \emph{depth} of a cluster tree.  
\end{definition}

\begin{definition}[block cluster tree, sparsity constant and partition]\label{def:partition}
		Let $\mathbb{T}_ \mathcal{I} $ be a cluster tree with root $\mathcal{ I}$ and $\eta > 0$ be a fixed admissibility parameter. 
		The block cluster tree $\mathbb{T}_{\mathcal{I} \times \mathcal{I}}$ is a tree  constructed recursively from the  root $\mathcal{I} \times \mathcal{I}$  such that  for each  block $\tau \times \sigma \in \mathbb{T}_{\mathcal{I} \times \mathcal{I}}$    with $\tau,\, \sigma  \in \mathbb{T}_\mathcal{I} $, the set of sons of $\tau \times \sigma$ is defined as 
		\begin{align*}
		\mathcal{S} (\tau \times \sigma):= \left\{ \begin{array}{ll}
		\emptyset   & \quad  \mbox{if\, $\tau \times \sigma$\, is $\eta$-admissible or $\mathcal{S} (\tau)= \emptyset$ or $\mathcal{S} (\sigma)= \emptyset$}, \\
		\mathcal{S} (\tau) \times \mathcal{S} (\sigma)  & \quad \mbox{else}.
		\end{array} \right.
		\end{align*}
		    The \emph{sparsity constant} $C_{\rm sp}$ of a block cluster tree, see, e.g., \cite{HackbuschKhoromskij2000a,grasedyck01}, is given as 
    \begin{equation}\label{eq:sparsityConstant}
	C_{\rm sp} := \max\left\{\max_{\tau \in \mathbb{T}_{\mathcal{I}}} \abs{\{\sigma \in \mathbb{T}_{\mathcal{I}} \, : \, \tau \times \sigma \in \mathbb{T}_{\mathcal{I} \times \mathcal{I}}\}},
	\max_{\sigma \in \mathbb{T}_{\mathcal{I}}}\abs{\{\tau \in \mathbb{T}_{\mathcal{I}} \, : \, \tau \times \sigma \in \mathbb{T}_{\mathcal{I} \times \mathcal{I}}\}}\right\}.
	\end{equation}
	The leaves of the block cluster tree induce a partition $P$ of the set $\mathcal{I} \times \mathcal{I}$, which we call a partition based on $\mathbb{T}_ \mathcal{I}$. For such a partition $P$ 
	and a fixed admissibility parameter $\eta > 0$, we define the \emph{far field} and the \emph{near field} 
	as 
	\begin{equation}\label{eq:farfield}
	P_{\rm far} := \{(\tau,\sigma) \in P \; : \; (\tau,\sigma) \; \text{is $\eta$-admissible}\}, 
	\quad P_{\rm near} := P\setminus P_{\rm far}.
	\end{equation}
\end{definition}

For clusters $\tau$, $\sigma \subset \mathcal{I}$, we adopt the notation 
\begin{align*}
\C^{\tau}& :=\{\Bx \in \C^{N}\colon \Bx_i = 0 \quad \mbox{ if $i \not\in \tau$}\}, \\
\C^{\tau\times\sigma}  & := 
  \{ \BA \in \C^{N \times N}\colon \BA_{ij} = 0 \quad \mbox{ if $i \not\in \tau$ or $j \not\in \sigma$}
   \}.
\end{align*}
For $\Bx \in \C^N$ and $\BA \in \C^{N \times N}$, the restrictions $\Bx|_\tau$ and $\BA|_{\tau\times\sigma}$ are understood as 
$(\Bx|_\tau)_i = \chi_\tau(i) \Bx_i$ and 
$(\BA|_{\tau\times\sigma})_{ij} =\chi_{\tau}(i) \chi_{\sigma}(j) \BA_{ij}$, where $\chi_\tau$ and $\chi_\sigma$ are the 
characteristic functions of the sets $\tau$, $\sigma$. For integers $r \in \mathbb{N}$, matrices $\C^{\tau\times r}$
are understood as matrices in $\C^{N \times r}$ such that each column is in $\C^\tau$. 

\begin{definition}[$\mathcal{H}$-matrices]
	Let $P$ be a partition of $\mathcal{I} \times \mathcal{I}$ based on a cluster tree $\mathbb{T}_{\mathcal{I}}$ 
	and admissibility parameter $\eta >0$. 
	A matrix $\mathbf{A}  \in \mathbb{C} ^ {N\times N}$ is an $\mathcal{H}$-matrix, if, for every
	admissible pair $(\tau , \sigma) \in P_{\rm far}$, we have a rank $r$ factorization 
	$$\mathbf{A}|_{\tau \times \sigma}=\mathbf{X}_{\tau \sigma}  \mathbf{Y}^ H_{\tau \sigma},$$
	where $\mathbf{X}_{\tau \sigma} \in \mathbb{C} ^ {\tau\times r}$ and  
	$\mathbf{Y}_{\tau \sigma} \in \mathbb{C} ^ {\sigma \times r}$. 
\end{definition}
%_____________________________________________________________________________________________________________________________________________________________
\subsection{Main result}
The following theorem is the main result of this paper. It states that the inverse of the Galerkin matrix $\BA$ from (\ref{eq:stiffness-matrix}) 
can be approximated at an exponential rate in the block rank by an $\CH$-matrix.
\begin{theorem}\label{th:H-Matrix approximation of inverses}
Let  $\eta >0$ be a fixed admissibility parameter and $P$ be a  partition 
	 of $\mathcal{I} \times \mathcal{I}$ based on the cluster tree $\mathbb{T}_{\mathcal{I}}$ and $\eta$. Let the mesh $\mathcal{T}_h$ be such that Assumption~\ref{def:mesh-conforming} holds true for any box.
Let $h< h _0$ with  $h_0$ given by Proposition~\ref{Lem. stability}, and let $\BA$ be the stiffness matrix
given by (\ref{eq:stiffness-matrix}). 
	Then, there exists an $\mathcal{H}$-matrix $\mathbf{B}_{\mathcal{H}}$ 
	with blockwise rank $r$ such that 
	\begin{align*}
	\left\|\mathbf{A}^{-1} -\mathbf{B}_{\mathcal{H}}
	\right\|_2 \le C_{\rm  apx} C_{\rm  sp} \operatorname*{depth}(\mathbb{T}_{\mathcal{I}}) 
	h ^{-1} e^{-b(r^{1/4}/\ln r)}.
	\end{align*}
	The constants
	$C_{\rm  apx}$, $b >0$  depend only on 
$\kappa$,	$\Omega$, $\eta$, and the $\gamma$-shape regularity of the quasi-uniform triangulation 
	$\mathcal{T}_h$. The constant $C_{\rm sp}$ (defined in (\ref{eq:sparsityConstant})) depends only on the partition $P$.
\end{theorem}
\begin{remark}
The low-rank structure of the far-field blocks allow for efficient storage of $\mathcal{H}$-matrices as the 
memory requirement to store an $\mathcal{H}$-matrix
is $\mathcal{O}(C_{\rm sp}\operatorname*{depth}(\mathbb{T}_{\mathcal{I}})r N)$. Standard clustering methods such as the 
geometric clustering for quasi-uniform meshes (see, e.g., \cite[Sec.~{5.4.2}]{hackbusch2015hierarchical})
lead to balanced cluster trees, i.e., $\operatorname*{depth}(\mathbb{T}_{\mathcal{I}})\sim \log(N)$ and a uniformly (in the mesh size $h$) 
bounded sparsity constant. In total this gives a storage complexity 
of $\mathcal{O}(rN\log(N))$ for the matrix $\mathbf{B}_{\mathcal{H}}$ 
instead of the $\mathcal{O}(N^2)$ for the fully populated inverse $\mathbf{A}^{-1}$. 
\eremk
\end{remark}
%_____________________________________________________________________________________________________________________________________________________________
\section{Helmholtz decompositions: continuous and localized discrete}
\label{sec:helmholtz-decomposition}
%---------------------------------
Helmholtz decompositions, i.e., writing a vector field as a sum of a 
divergence-free field and a gradient field, play a key role in our analysis. In fact, we use two different decompositions, 
the   regular  decomposition (see, e.g., \cite[Lem.~{2.4}]{hiptmair2002finite}  and \cite[Thm.~{11}]{hiptmair2015maxwell}) 
and a localized discrete version (Definition~\ref{The local  Helmholtz decomposition}).
\begin{lemma}[Regular decomposition]
	\label{Helmholtz decomposition}
	Let $\Omega\subset {\mathbb R}^3$ be a bounded Lipschitz domain. 
	Then, there is a constant $C > 0$ depending only on $\Omega$ such that 
	any $\BE  \in \BH_0({\operatorname{curl}},\Omega)$ can be written as 
	$\BE = \Bz + \nabla p$ with $\Bz \in \BH^1_0(\Omega)$ and 
	$p \in H^1_0(\Omega)$  and 
	\begin{align*} 
	\|\Bz\|_{\BH _0^1(\Omega)} & \leq C \|\BE\|_{\BH (\operatorname{curl},\Omega)}, 
	&
	\|\Bz\|_{\BLT(\Omega)} +  
	\|\nabla p \|_{\BLT(\Omega)} & \leq 
	C \|\BE\|_{\BLT(\Omega)}. 
	\end{align*}
\end{lemma}
\begin{proof} 
	Regular decompositions are available in the literature, see, e.g., \cite[Lem.~2.4]{hiptmair2002finite} and  \cite[Thm.~11]{hiptmair2015maxwell}.
	The statement that $\|\Bz\|_{\BLT(\Omega)}$ and 
	$\|\nabla p \|_{\BLT(\Omega)}$ are controlled by 
	$\|\BE\|_{\BLT(\Omega)}$ is a variation of these estimates. For a proof, 
	see \cite{HP20} or Appendix~\ref{sec:appendix}. 
\qed
\end{proof}
	The function $\Bz$ of the regular decomposition provided by Lemma~\ref{Helmholtz decomposition} is not necessarily 
	divergence-free. This can be corrected by subtracting a gradient. To that end, we introduce, for a given open set 
	$\widetilde{D} \subseteq \Omega$ and a chosen $\widetilde \eta \in L^\infty(\Omega)$ with $\widetilde \eta \equiv 1$ on $\widetilde{D} $, the 
	mapping $\BLT(\Omega) \rightarrow H^1_0(\Omega)\colon \Bz \mapsto \varphi_{\Bz}$ by 
	\begin{equation}
	\label{eq:varphi_z}
	\langle \nabla \varphi_{\Bz}, \nabla v\rangle_{L^2(\Omega)} = \langle \widetilde \eta \Bz, \nabla v\rangle_{L^2(\Omega)} 
	\qquad \forall v \in H^1_0(\Omega). 
	\end{equation}
	\begin{lemma}
		\label{lemma:varphi_z}
		The mapping $\BLT(\Omega) \ni \Bz \mapsto \varphi_{\Bz} \in H^1_0(\Omega)$ has the following properties: 
		\begin{enumerate}[(i)]
			\item 
			\label{item:lemma:varphi_z-i}
			$\norm{\varphi_{\Bz}}_{H^1(\Omega)} \leq C \|\widetilde \eta\|_{L^\infty(\Omega)} \norm{\Bz}_{\BLT(\operatorname{supp} \widetilde \eta)}$, where the constant 
			depends only on $\Omega$. 
			\item 
			\label{item:lemma:varphi_z-ii}
			$\langle \Bz - \nabla \varphi_{\Bz}, \nabla v \rangle_{L^2(\widetilde{D})} = 0$ for all $ v \in H^1_0(\Omega)$.
			 
		\end{enumerate}
	\end{lemma}%
	\begin{proof}
		By construction, we have $\norm{\nabla \varphi _z}_{\BLT(\Omega)} \leq \norm{\widetilde \eta \Bz}_{\BLT(\Omega)}$. The constant $C$ 
		in statement (\ref{item:lemma:varphi_z-i}) reflects the Poincar\'e constant of the simply connected domain $\Omega$. The property 
		(\ref{item:lemma:varphi_z-ii}) follows by construction. %
\qed
	\end{proof}
\begin{remark}[classical Helmholtz decomposition] 
\label{remk:classical-helmholtz-decomposition}
Selecting $\widetilde D = \Omega$ and correspondingly $\widetilde \eta \equiv 1$ yields 
the decomposition $\BE  = (\Bz - \nabla \varphi_{\Bz}) + \nabla (p + \varphi_{\Bz})$ with the orthogonality $\langle \Bz - \nabla \varphi_{\Bz}, \nabla (p+\varphi_{\Bz})\rangle_{L^2(\Omega)} = 0$ 
and $\|\Bz - \nabla \varphi_{\Bz}\|_{\BH(\operatorname{curl},\Omega)} \lesssim \|\BE\|_{\BH(\operatorname{curl},\Omega)}$,  
$\|\Bz - \nabla \varphi_{\Bz}\|_{\BLT(\Omega)} \lesssim \|\BE\|_{\BLT(\Omega)}$, $\|\nabla (p+\varphi_{\Bz})\|_{L^2(\Omega)} \lesssim \|\BE\|_{\BLT(\Omega)}$.
\eremk
\end{remark}
Regular decompositions as in Lemma~\ref{Helmholtz decomposition} can also be done locally for discrete functions. 
%Let $\mathcal P _ m (T)$ denote the space of polynomials of degree at most $m \in \mathbb{N}_0$ on $T \in \TT _h.$
Let $\mathcal P _ 1 (T)$ denote the space of polynomials of degree at most $1$ on $T \in \TT _h.$
We introduce spaces of globally continuous, piecewise linear polynomials by
\begin{align}
S ^ {1,1}(\T _h) & : = \lbrace  p_h \in H^1  (\Omega ) \; \colon \;  p_h | _T \in \mathcal P _ 1(T) \quad \forall T \in \T_h \rbrace, \\
S_0^{1,1}(\T _h) &:= S^{1,1}(\T _h) \cap H^1_0(\Omega).
\end{align}
%and the \emph{localized} spaces as follows.
We will require the following assumption on the meshes $\T_h$:

\begin{assumption}\label{def:mesh-conforming}
For a simply connected domain $D \subset \R^3$, define the sets of elements touching $D$ as
\begin{align*}
\TT_h(D)& :=\{T \in \T_h  \;\colon \;\abs{T \cap D} >0 \}, \\
		\widehat{ D}&:=\operatorname{int} \Bigl( \bigcup\limits_{T \in \TT_h(D)} \overline{T}\Bigr). 
\end{align*} 
For any box $D \subset \R^3$ there is a set 
$\widetilde{D}$, which is a union of elements in $\T_h$ such that 
\begin{enumerate}
\item $\widehat D \subset \widetilde D$,
\item $\operatorname*{dist} (\partial \widetilde D, D) \leq 2h$,
\item $\widetilde D$ is simply connected.
\end{enumerate}
We call $\widetilde{D}$ a \emph{mesh-conforming region} for 
$D$. If a box $D$ has more than one mesh-conforming region $\widetilde{D}$, 
one is selected as ``the'' mesh-conforming one.
\eremk
\end{assumption}

\begin{remark}
The reasoning behind Assumption~\ref{def:mesh-conforming} is that the region $\widehat{D}$ may not be simply connected, but by adding elements of the mesh 
holes may be filled to obtain a simply connected set $\widetilde{D}$. 
\eremk
\end{remark}

 The spaces \emph{localized} to a mesh-conforming region $\widetilde{D}$ are given by 
\begin{align}
\label{eq:local-S11}
S^{1,1}(\TT_h,\widetilde{D}) &:= \lbrace  p_h|_{\widetilde{D}}\;\colon \; p_h \in S^{1,1}_0(\T_h)\rbrace, \\ 
\BX_h(\TT_h,\widetilde{D}) &:= \lbrace  \BE_h|_{\widetilde{D}}\;\colon \; \BE_h \in \BX_{h.0}(\T_h, \Omega)\rbrace.
\end{align}

	\begin{definition}[Local discrete regular decomposition]\label{The local  Helmholtz decomposition}
\label{def:local-helmholtz}
Let $D\subset \R^3$ be a box and $\widetilde D$ be the corresponding mesh-conforming region. 
We denote by $\Pi^\nabla_{\widetilde D}:\BLT(\widetilde D) \rightarrow \nabla S^{1,1}(\TT_h,\widetilde{D})$ the 
$\BLT(\widetilde{D})$-projection onto $\nabla S^{1,1}(\TT_h, \widetilde{D})$ given by 
\begin{equation}
\label{eq:local-div}
\langle {\mathbf p} - \Pi^\nabla_{\widetilde{D}} {\mathbf p}, \nabla v_h\rangle_{\BLT(\widetilde{D})} = 0 
\qquad \forall v_h \in S^{1,1}(\TT_h,\widetilde{D}). 
\end{equation}
	Let $\eta \in C^\infty (\overline{\Omega}) $ be a cut-off function with 
$0 \le \eta \le 1$ and $\eta \equiv 1$ on $\widetilde D$. 
Let $\BE_h$ be such that $\eta \BE_h \in \BH_0(\operatorname{curl},\Omega)$ as well as  
$\BE_h|_{\widetilde D} \in \BX_{h}(\T_h,\widetilde{D})$. 
Decompose 
  $\eta \BE _h \in \BH_0(\operatorname*{curl},\Omega) $  as 
 $\eta \BE_h =  \Bz   +  \nabla  p$, 
where
$ \Bz \in \BH _0 ^1 (\Omega) $ and $ p  \in H_0^1 (\Omega)$ are given by Lemma~\ref{Helmholtz decomposition}. 

Then, the \emph{local discrete regular decomposition} is given by 
$\BE_h = \Bz_h + \Pi^\nabla_{\widetilde{D}} \nabla p$ on $\widetilde{D}$ with $\Bz_h:= \BE_h - \Pi^\nabla_{\widetilde{D}}\nabla p$. 
We write $\nabla p_h = \Pi^\nabla_{\widetilde{D}} \nabla p$ for some $p_h \in S^{1,1}(\TT_h,\widetilde{D})$. 
	\end{definition}
For future reference, we note that 
\begin{equation}
\label{eq:L2-projection}
\norm{\Pi^{\nabla}_{\widetilde{D}} {\mathbf p} }_{\BLT(\widetilde{D})} \leq \norm{{\mathbf p}}_{\BLT(\widetilde{D})}. 
\end{equation}
\begin{remark}
\begin{enumerate}
\item 
The function $p_h  \in S^{1,1}(\TT_h,\widetilde{D})$ that satisfies $\nabla p_h = \Pi^\nabla_{\widetilde{D}}  {\mathbf p}$,  
is not unique. However, its gradient $\nabla p_h$ is unique. 
\item 
Due to the cut-off function $\eta$, the decomposition depends on $\BE_h$ on $\operatorname{supp} \eta$ only, which is quantified in 
the stability assertions of Lemma~\ref{local-stability-approx-divergence-free}. 
\item
The local regular decomposition provides, for a function $\BE_h$ that is a discrete function on $\widetilde{D}$, 
two representations in view of $\eta \equiv 1$ on $\widetilde{D}$, namely, 
$\BE_h| _ {\widetilde{D}} = \left( \Bz + \nabla p \right) | _ {\widetilde{D}}= \Bz_h + \nabla p_h$.  
\item 
For $\BE_h \in \BX_{h,0}(\T_h,\Omega)$, the decomposition 
$\BE_h = (\Bz - \nabla \varphi_{\Bz}) + \nabla (p + \varphi_{\Bz})$ of Remark~\ref{remk:classical-helmholtz-decomposition} yields upon setting 
$\nabla p_h:= \Pi^\nabla_\Omega \nabla (p+\varphi_{\Bz}) \in \nabla S^{1,1}_0(\T_h,\Omega)\subset \BX_{h,0}(\T_h,\Omega)$ and $\Bz_h:= \BE_h -\nabla p_h \in \BX_{h,0}(\T_h,\Omega)$ 
the decomposition $\BE_h = \Bz_h + \nabla p_h$ with 
\begin{align*}
\langle \Bz_h,\nabla p_h\rangle_{\BLT(\Omega)} = 0, 
\qquad 
\|\Bz_h\|_{\BLT(\Omega)} + \|\nabla p_h\|_{\BLT(\Omega)} &\lesssim \|\BE_h\|_{\BLT(\Omega)}, 
\qquad \\
\|\Bz_h\|_{\BH(\operatorname{curl},\Omega)} &\lesssim 
\|\BE_h\|_{\BH(\operatorname{curl},\Omega)},  
\end{align*}
which is a discrete Helmholtz decomposition as described in, e.g., \cite[Cor.~{5.1}]{girault2012finite} and \cite[Sec.~7.2.1]{monk2003finite}. 
\eremk
\end{enumerate}
\end{remark}
The following lemma formulates a local exact sequence property. 

\begin{lemma} \label{discrete potentials}
	Let $D \subset \R^3$ be a box such that $D\cap \Omega$ is a simply connected Lipschitz domain  and $\widetilde{D}$ 
	be given according to Assumption~\ref{def:mesh-conforming}. Assume that 
	 $\overline{\widetilde D} \cap \partial\Omega$ is connected. 
	 (In particular, the empty set is connected.)
	Then, for all  ${\mathbf v}_h \in {\mathbf X}_h(\TT_h,\widetilde{D})$ with
$\operatorname{curl}{\mathbf v}_h = 0$ on $\widetilde{D}$, 
we can find  a $\widetilde \varphi _h \in S ^ {1,1}(\TT_h,\widetilde{D} )$ such that $\boldsymbol{\mathrm{v}}_h = \nabla \widetilde \varphi _h$.
\end{lemma}
	\begin{proof}
		We recall from, e.g., \cite[Thm.~{3.37}]{monk2003finite} the following commuting diagram property: 
		for a simply connected Lipschitz domain $\omega$ the 
		condition $\nabla \times \Bw = 0$ implies $\Bw = \nabla \psi$ for some $\psi \in H^1(\omega)$; furthermore, $\psi$ is unique
		up to a constant. The discrete commuting diagram property for a tetrahedron $T$ is: if $\Bw \in {\mathcal N}_0(T)$ satisfies 
		$\nabla \times \Bw = 0$, then there is $\psi_h \in {\mathcal P}_1(T)$ with $\Bw = \nabla \psi_h$.

 The condition $\nabla \times \Bv _h = 0$ on $\widetilde{D}$ implies $\Bv _h = \nabla \varphi_h$ for some 
		$\varphi_h \in H^1(\widetilde{D})$. 
		The function $\varphi_h$ is unique up to a constant, which we fix, for example,  by the condition 
		$\int_{\widetilde{D}} \varphi_h = 0$. 
		For each $T \in \TT_h(D)$, the condition $\nabla \times \Bv _h = 0$ on $T$ implies the existence of $\widetilde \varphi_{h,T} \in {\mathcal P}_1(T)$
		with $\Bv _h = \nabla \widetilde\varphi_{h,T} $ on $T$. The polynomial $\widetilde\varphi_{h,T}$ is unique up to a constant, 
		which we fix by requiring $\int_{T} \widetilde \varphi_{h,T} = \int_{T} \varphi_{h}$. By the uniqueness assertion
		we have $\varphi_h|_{T} = \widetilde \varphi_{h,T}|_{T}$. 		
		Define $\widetilde \varphi_h \in S^{1,0}(\TT_h,\widetilde{D})$ elementwise by $\widetilde \varphi_h|_T = \widetilde \varphi_{h,T}$. Since $\varphi_h \in H^1(\widetilde{D})$ we directly obtain $\widetilde \varphi_h \in S ^ {1,1}(\TT_h,\widetilde{D}) $.

\qed
	\end{proof}

In order to prove the following lemmas, we need to introduce some projections and  their properties.
		Let $D \subset \R^3$   be a box and $\widetilde{D}$ 
be defined according to Assumption~\ref{def:mesh-conforming}. We define the space
\begin{align*}
\BH(\operatorname*{div},\widetilde{D}):= \left\lbrace \textbf{U} \in \BLT (\widetilde{D} ) \;\colon \;\nabla   \cdot \textbf{U}\in L^2 (\widetilde{D} )  \right\rbrace. 
\end{align*}
Let  $\bold{RT}_0(T):= \{  \mathbf{a}+ b \Bx \;\colon \;\mathbf{a}  \in \mathbb{R}^3 
	, \, b  \in \mathbb{R}  \}$ be the classical lowest order  Raviart-Thomas element defined on $T$. Introduce 
\begin{equation}
\label{eq:space-Vh}
\mathbf{V}_h(\TT_h, \widetilde{D}):= \{ \textbf{U}_h \in \BH(\operatorname*{div},\widetilde{D}) \;\colon \; \textbf{U}_h | _ T \in \bold{RT}_0(T) \quad \forall T \in \TT_h(D)\}.
\end{equation}
	On $\widetilde{D}$ the Raviart-Thomas interpolation operator 
$\Bw_{\widetilde{D}}\colon \BH ^1 (\widetilde{D})\rightarrow  \mathbf{V}_h(\TT_h,\widetilde{D})$  is defined elementwise by 
	$\Bw _{\widetilde{D}}\BU|_T:=\Bw _T \BU $, where
	the elemental interpolation operator  $ \Bw _T \colon \BH^1(T) \rightarrow \bold{RT}_0(T)$ is characterized by the vanishing of certain moments of $\BU-\Bw _T \BU$, viz., 
	$$ 
        \int_{f} (\BU-\Bw _T \BU) \cdot  \nu q \,d A =0 \quad \forall q \in \mathcal P _0 (f)\,\, \forall\; \text{faces $f$ of} \,\,T \in \TT_h,  
        $$
	where $\nu$ is the unit normal  to $f$ and $d A $ denotes the surface measure on $f$.
Define the space 
\begin{align}
\label{eq:space-Dh}
\BD_h(\TT_h,\widetilde D):=\{\BU \in \BH^1(\widetilde{D})\;\colon \; \nabla \times \BU \in \BH^1(T) \quad \forall T \in \TT_h(D)\},
\end{align}
and the N\'ed\'elec interpolation operator $\Br_{\widetilde{D}}: \BD _ h (\TT_h,\widetilde{D})\rightarrow  \BX_{h}(\TT_h ,\widetilde{D})$   elementwise by 
	$\Br _{\widetilde{D}}\BU|_T:=\Br _T \BU $, where
the elemental interpolant $ \Br _T \BU \in  \CN _0 (T)$ is characterized by the vanishing of certain moments of $\BU-\Br _T \BU$, viz., 
	$$
\int_{e}\left( \BU-\Br _T \BU \right) \cdot \boldsymbol{\tau}\, de =0  \quad \forall\; \text{edges $e$ of }\,T \in \TT_h; 
        $$
	here, $\boldsymbol{\tau} $ is a unit vector parallel to the edge $e$.
        A key property of the operators $\Br _{\widetilde{D}}$ and $\Bw _{\widetilde{D}}$ is that they commute, i.e.,
(see, e.g., \cite[(5.59)]{monk2003finite}) 
	\begin{equation}
\label{eq:commuting-diagram}
\Bw_{\widetilde{D}} \nabla  \times \bold{ U}  = \nabla \times \Br_{\widetilde{D}} \bold{ U}   \qquad \forall\; \bold{ U}  \in \BD_h(\TT_h,\widetilde{D}).
\end{equation}
Moreover, the lowest order elemental  N\'ed\'elec interpolants have first order approximation properties.

	\begin{lemma}\cite[Thm.~{5.41}]{monk2003finite}\label{approximation properties of the local interpolation operator}
\label{lemma:monk-thm.5.41}
		Let $T \in \T_h$. Then, for $ \BU \in \BH^1(T) $ with $\nabla \times \BU \in \BH^1(T)$, we have
		\begin{align*}
		\norm{\BU - \Br _T\BU}_{\BLT(T)}& \lesssim h \left(\abs{\BU}_{\BH^1 (T)}+\norm{\nabla \times \BU}_{\BH^1 (T)}  \right), \\
		\norm{\nabla \times (\BU - \Br _{T} \BU)}_{\BLT(T)}& \lesssim h \norm{\nabla \times \BU}_{\BH^1 (T)}.  
		\end{align*}
	\end{lemma}
In the following, we show local stability and approximation properties for the local discrete regular decomposition of Definition~\ref{def:local-helmholtz}. 
This will be based on Lemma~\ref{discrete potentials} with $D = B_R$, where $B_R$ is a box with side length $R$. 
It is an important geometric observation that, due to the assumption that $\Omega$ is a Lipschitz polyhedron,
the intersection $B_R \cap \Omega$ is a Lipschitz domain  and the intersection $B_R \cap \Omega^c$ is connected provided $R$ is sufficiently small.  
Then, the additional assumptions on $D \cap\Omega = B_R \cap \Omega$ in Lemma~\ref{discrete potentials} can be satisfied.
We formulate this as an assumption on $R$ in terms of a number $R_{\rm max}$ that depends on $\Omega$:

\begin{definition}[$R_{\rm max}$]  
\label{def:Rmax}
$R _ {\rm max}>0$ is such that for any $R \in (0,R _ {\rm max}]$ and any box $B_ R$ with $\abs{B_R \cap \Omega}>0$, the intersection $B_R \cap \Omega$ is a Lipschitz domain and $B_R \cap \Omega ^c$ is connected.
\end{definition}
 
\begin{lemma}[stability of local discrete regular decomposition]
\label{local-stability-approx-divergence-free}
		Let   $\varepsilon \in (0,1)$, $R \in (0,R _ {\rm max}]$ be  such that  $\frac{h}{R} < \frac{\varepsilon}{4}$, and let 
  $B_R$ and $B_{(1+\varepsilon)R}$ be concentric boxes. Define 
$\widetilde B _R$ 
    according to Assumption~\ref{def:mesh-conforming}. 
Let 
$\eta \in W^{1,\infty}(\Omega) $ be a  cut-off function with $\operatorname{supp} \eta \subseteq \overline{B_{(1+\varepsilon)R} \cap \Omega}$,   
$\eta \equiv 1$ on $\widetilde B_{R}$, $0 \le \eta \le 1$, and  $ \left\| \nabla \eta  \right\|_{L^\infty(\Omega)} 
\leq C_\eta \frac{1}{ \varepsilon R}$. Let $\BE_h \in \BH(\operatorname{curl}, B_{(1+\varepsilon)R}\cap\Omega)$ 
be such that $\eta \BE_h \in \BH_0(\operatorname{curl},\Omega)$ as well as $\BE_h \in \BX_h(\T_h,\widetilde{B}_{R})$. 
Let $\eta \BE _h  = \Bz + \nabla p$ be the regular decomposition of $\eta \BE_h$ given by Lemma~\ref{Helmholtz decomposition}
and let $\Bz_h$ and $\nabla p_h$ be the contributions of the local discrete regular decomposition of Definition~\ref{def:local-helmholtz} 
with $D = B_{R}$ and $\widetilde{D} = \widetilde{B}_{R}$ there. 
		 Then, $\BE_h = \Bz_h + \nabla p_h$  on $\widetilde{B}_{R}\cap\Omega $, and the following  local stability and  approximation   results  hold: 
		 \begin{align*}
		 	\norm{\nabla p_h}_ {\BLT (B _R \cap \Omega)}+	 \norm{\Bz _h }_ {\BH  (\operatorname*{curl},B _R \cap \Omega)} & \le C\left( \norm{ \nabla \times  \BE _h}_ {\BLT (B _{(1+\varepsilon)R} \cap \Omega)}+\frac{ 1}{\varepsilon R} \norm{ \BE _h}_{\BLT (B _{(1+\varepsilon)R}\cap \Omega)} \right) ,\\
		 \norm{\Bz - \Bz _h} _ {\BLT (B _R \cap \Omega)} & \le C h \norm{\Bz}_ {\BH^1 (B_{(1+\varepsilon)R}\cap \Omega)} \\
 & \le Ch\left( \norm{ \nabla \times  \BE _h}_ {\BLT (B_{(1+\varepsilon)R}\cap \Omega)}+\frac{ 1}{\varepsilon R} \norm{ \BE _h}_{\BLT (B _{(1+\varepsilon)R}\cap \Omega)} \right),
		 \end{align*}
		 where the constant $C >0$ depends only on $\Omega$, the $\gamma$-shape regularity of the quasi-uniform triangulation $\T _h$, and $C_\eta$. 
\end{lemma}
 \begin{proof}
 	The proof is done in two steps.
 	 We note that the condition on the parameter $\varepsilon$ and the assumption on the mesh conforming region (Assumption~\ref{def:mesh-conforming}) ensures that $\widetilde B_{R} \subseteq B_{(1+\varepsilon)R}$. 
 	
 {\bf Step 1:}   	
 In this step we provide a proof of the stability estimate.
Recalling the stability estimate 
Lemma~\ref{Helmholtz decomposition} and  using the  product rule for the $\operatorname*{curl}$ operator, it follows that
 \begin{align}\label{stability- continuous-divergence- free}
\nonumber  
& \norm{ \Bz} _ {\BH _0 ^1 (\Omega)}+\norm{\nabla  p} _{L^2  (\Omega)} \lesssim   \norm{  \eta \BE _h }_ {\BH  (\operatorname*{curl},\Omega)} 
\\ 
\nonumber 
& \qquad \lesssim  \norm{ \nabla \times  \BE _h}_ {\BLT (B _{(1+\varepsilon)R} \cap \Omega)}
+  \norm{\nabla \eta }_{L^\infty(B_{(1+\varepsilon)R}\cap \Omega)}\norm{ \BE _h}_{\BLT (B_{(1+\varepsilon)R} \cap \Omega)}+ \norm{ \BE _h}_{\BLT (B_{(1+\varepsilon)R} \cap \Omega)}\\
  & \qquad \stackrel{\varepsilon R \lesssim 1}{  \lesssim}  \norm{ \nabla \times  \BE _h}_ {\BLT (B _{(1+\varepsilon)R}\cap \Omega)}+\frac{ 1}{\varepsilon R} \norm{ \BE _h}_{\BLT (B _{(1+\varepsilon)R}\cap \Omega)}.
 \end{align}
 Since $\nabla p_h$ satisfies  \eqref {eq:local-div}, we get with \eqref{eq:L2-projection} and the aid of  \eqref{stability- continuous-divergence- free} 
 \begin{align*}
\norm{\nabla p_h} _{L^2  (B _R \cap \Omega)} \le \norm{ \nabla  p} _{L^2  (\widetilde{B}_{R})}\le \norm{ \nabla  p} _{L^2  (\Omega)} \lesssim \norm{ \nabla \times  \BE _h}_ {\BLT (B _{(1+\varepsilon)R} \cap \Omega)}+\frac{ 1}{\varepsilon R} \norm{ \BE _h}_{\BLT (B _{(1+\varepsilon)R} \cap \Omega)}. 
 \end{align*}
 The definition of $\Bz _h$ gives
 \begin{align*}
 \norm{\Bz _h }_ {\BH  (\operatorname*{curl},B_R \cap \Omega)} & \lesssim  \norm{ \nabla \times  \BE _h}_ {\BLT (B _{(1+\varepsilon)R} \cap \Omega)}+\frac{ 1}{\varepsilon R} \norm{ \BE _h}_{\BLT (B _{(1+\varepsilon)R} \cap \Omega)}. 
 \end{align*}
 The  combination of the above inequalities provides the desired local stability result.
 
 {\bf Step 2:} 
 To prove the approximation property, we first need to ascertain the existence of $\varphi _ h \in S^{1,1}(\TT_h,\widetilde{B}_R)$ such that 
 $\Bz _h - \Br_{\widetilde{B}_{R}} \Bz  = \nabla \varphi _h$ on $\widetilde B_R$. 
To that end, we note that $\Bz_h \in \BD_h(\TT _h,\widetilde{B}_{R})$, use   
the commuting diagram property (\ref{eq:commuting-diagram})   of $\Br_{\widetilde{B}_{R}}$ and $\Bw_{\widetilde{B}_{R}}$, and the fact that 
$\Br_{\widetilde{B}_{R}}$ is a projection operator  to compute on $\widetilde B_{R}$: 
\begin{align*}
\nonumber 
\nabla \times (\Bz _h - \Br_{\widetilde{B}_{R}} \Bz ) &= \nabla \times \Bz _h - \Bw_{\widetilde{B}_{R}} \nabla \times \Bz  
= \nabla \times ( \BE _h | _ {\widetilde B_{R}})  - \Bw_{\widetilde{B}_{R}} \nabla \times  (\BE _h | _ {\widetilde B_{R}}) \\
& =  \nabla \times ( \BE _h | _ {\widetilde B_{R}}) - \nabla \times \Br_{\widetilde{B}_{R}} (\BE _h | _ {\widetilde B_{R}})= 0.
\end{align*} 
Lemma~\ref{discrete potentials} then provides the existence of $\varphi _ h \in S^{1,1}(\TT_h,\widetilde{B}_R)$ such that 
$\Bz _h - \Br_{\widetilde{B}_{R}} \Bz  = \nabla \varphi _h$ on $\widetilde B_R$. 
Since $p_h$ satisfies  \eqref{eq:local-div}, we get from $\Bz + \nabla p = \BE_h = \Bz_h + \nabla p_h$ on $\widetilde{B}_R$ and 
the approximation property of $\Br_{\widetilde{B}_{R}}$ 
given in Lemma~\ref{lemma:monk-thm.5.41}
\begin{align*}
\nonumber 
\norm{\Bz - \Bz _h} ^2 _ {\BLT (\widetilde B _R)}& = \skp{\Bz -\Br_{\widetilde{B}_{R}} \Bz , \Bz - \Bz _h }_  {\BLT (\widetilde B_R)}  +
 \skp{ \Br_{\widetilde{B}_{R}}\Bz -\Bz_h , \Bz - \Bz _h }_  {\BLT ( \widetilde B_R)}\\
\nonumber &=  \skp{\Bz - \Br_{\widetilde{B}_{R}}  \Bz , \Bz - \Bz _h }_  {\BLT (\widetilde B_R)} - \skp{ \nabla \varphi _h , \nabla (p _h - p) }_  {\BLT (\widetilde B_R)}
 \\ \nonumber 
 & = \skp{\Bz - \Br_{\widetilde{B}_{R}}  \Bz , \Bz - \Bz _h }_  {\BLT (\widetilde B_R)} \lesssim \norm{\Bz - \Br_{\widetilde{B}_{R}} \Bz}  _ {\BLT (\widetilde B _R)} \norm{\Bz - \Bz _h}  _ {\BLT (\widetilde B _R)}\\
 & 
\lesssim  h \norm{\Bz}_ {\BH^1 ({B_{(1+\varepsilon)R}}\cap \Omega )}\norm{\Bz - \Bz _h}  _ {\BLT (\widetilde B_R)}.
\end{align*}
The combination of the above inequality and  \eqref{stability- continuous-divergence- free} implies
\begin{align*}
\nonumber \norm{\Bz - \Bz _h}  _ {\BLT ( B _R \cap \Omega) } & \le \norm{\Bz - \Bz _h} _ {\BLT (\widetilde B _R)}  \lesssim  h \norm{\Bz}_ {\BH^1 ({B_{(1+\varepsilon)R}} \cap \Omega)}\\
&\lesssim h \left( \norm{ \nabla \times  \BE _h}_ {\BLT (B _{(1+\varepsilon)R} \cap \Omega)}+\frac{ 1}{\varepsilon R} \norm{ \BE _h}_{\BLT (B _{(1+\varepsilon)R}\cap \Omega)}\right) ,
\end{align*}
which finishes the proof.
\qed
 \end{proof}

%%%%%%%%%%%%%%%%%%%%%%%%%%%%%%%%%%%%%%%%%%%%%%%%%%%%%%%%%%%%%%%%%%%%%%%%%%%%%%%%%%%%%%%%%%%%%%%%%%%%%%%%
\section{Low-dimensional approximation of discrete ${\boldsymbol {\mathcal L}}$-harmonic functions} 
\label{sec:low-dimensional} 
We say that $\BE_h \in \BX_h(\T_h,\widetilde D)$ is \emph{discrete ${\boldsymbol{\mathcal L}}$-harmonic} on $\widetilde{D}$, if  
$a(\BE_h,\Bv_h) = 0$ for all $\Bv_h \in \BX_{h,0}(\T_h,\Omega)$ with $\operatorname{supp} \Bv_h \subset \overline{\widetilde{D}}$; such a space 
will be formally introduced as $\CH_{c,h}(\widetilde{D})$ below. 
In this section, we show that discrete ${\boldsymbol{\mathcal L}}$-harmonic functions can be approximated from low-dimensional spaces on compact subsets of $\widetilde{D}$. 
Discrete interior regularity estimates, introduced in the following, play a key role. 
%________________________________________________________________________________________________________________________________________________________________
\subsection{The Caccioppoli-type inequalities}

Caccioppoli inequalities usually estimate higher order derivatives by lower order derivatives on (slightly) enlarged regions.
The following discrete Caccioppoli-type inequalities are formulated with an $h$-weighted $\BH(\operatorname*{curl})$-norm 
and an $h$-weighted $H^1$-norm. 
For a box $B_R$ of side length $R>0$, we define 
the norms $\vertiii{\cdot}_{c,h,R}$ and $\vertiii{\cdot}_{g,h,R}$ (the subscripts $c$ and $g$ abbreviate `curl' and `gradient') as follows:  
\begin{alignat}{3}
\label{eq:triple-norm-c}
\vertiii{\BU} ^2 _{c,h, R} &:=
\frac{h^2}{R^2}\left\| \nabla \times \BU \right\|^2_{\BLT(B_{R} {\cap \Omega})} 
+\frac{1}{R^2} \left\| \BU \right\|^2_{\BLT(B_{R}{\cap \Omega})} 
&& \qquad \forall \BU \in \BH(\operatorname*{curl},B_R \cap \Omega),   \\
\label{eq:triple-norm-g}
\vertiii{u} ^2 _{g,h, R} & :=
\frac{h^2}{R^2}\left\| \nabla u \right\|^2_{\BLT(B_{R}{\cap \Omega})} 
+\frac{1}{R^2} \left\| u \right\|^2_{L^2(B_{R}{\cap \Omega})}
&& \qquad \forall 
u \in H^1(B_R \cap \Omega) . 
\end{alignat}
 For any bounded open set $B \subset \R^3$, we define
 \begin{align*}
 \mathcal{H}_{c,h}(B \cap \Omega):= \lbrace \BU _h \in \BH(\operatorname*{curl},B \cap \Omega)\;\colon\; &  \exists \widetilde \BU _h \in \BX _ {h,0}(\T _h,\Omega )\,\, 
                               \text{s.t.}\,\, \BU _h | _ {B \cap \Omega} =  \widetilde \BU _h |_{B \cap \Omega}, \\
 &  a(\BU_h,\Psi_h) = 0 \quad \forall \Psi _ h \in    \BX _ {h,0}(\T _h, \Omega) , \,
 \operatorname*{supp} \Psi_h \subset \overline{B \cap \Omega} \rbrace 
\end{align*}
and 
 \begin{align*}
 \mathcal{H}_{g,h}(B \cap \Omega):= \lbrace p_h \in H^1  (B \cap \Omega) \;\colon & \; \exists \widetilde p _h \in S _0^ {1,1}(\T _h)\,\, \text{s.t.}\,\, p _h | _ {B \cap \Omega} =  \widetilde p _h |_{B \cap \Omega}, \\
&  \skp{\nabla p _h, \nabla  \psi _h}_ {\BLT (B \cap \Omega)}=0
 \forall \psi _ h \in    S _0^ {1,1}(\T _h), \,
 \operatorname*{supp} \psi_h \subset \overline{B \cap \Omega} \rbrace .
 \end{align*}
The following lemma provides a discrete Caccioppoli-type estimate for functions
in \linebreak $\CH _{c,h} (B _{(1+\varepsilon)R} \cap \Omega)$.
 \begin{lemma}\label{th:Caccioppoli-divergence free part}
 	Let   $\varepsilon \in (0,1)$ and $R \in (0,2\operatorname*{diam}(\Omega))$ be  such that  $\frac{h}{R} < \frac{\varepsilon}{4}$.  Let  
 	 $B_R$ and $B_{(1+\varepsilon)R}$ be two concentric boxes and $\BE _h \in \mathcal{H}_{c,h}(B_{(1+\varepsilon)R}\cap \Omega)$.
 	Then, there exists a constant $C$ depending only on 
 	$\kappa$, $\Omega$,  
 	and the $\gamma$-shape regularity of the quasi-uniform triangulation $\T_h$ such that 
 	\begin{align}
 	\nonumber
 	\left\| \nabla \times  \BE _h \right\|_ {\BLT(B_{R} \cap \Omega)}
 	&\le 
 	C \frac{1+\varepsilon}{\varepsilon} \vertiii{\BE _h}_{c,h,(1+\varepsilon) R} . 
 	\end{align}
 \end{lemma}
 \begin{proof}
 Let  $\eta \in  C^\infty(\overline{\Omega}) $ be a cut-off function with $\operatorname{supp} \eta \subseteq B_{(1+\varepsilon /2)R} $,  $0 \le \eta \le 1$, $\eta \equiv 1$ on $B_R \cap \Omega$, 
and  $\|\nabla^j \eta\|_{L^\infty(\Omega)} \lesssim (\varepsilon R)^{-j}$ for $j \in \{0,1,2\}$. 
We notice $\operatorname{supp} (\eta ^ 2 \BE _h ) \subseteq \overline{B_{(1+\varepsilon /2)R}\cap {\Omega}}$ and since $4h \le \varepsilon R $ 
we have  $\operatorname{supp} \, \Br_\Omega(\eta ^ 2 \BE _h ) \subseteq \overline{B_{(1+\varepsilon )R}\cap {\Omega}}$.
 The proof is done in two steps.\\
 \noindent

 {\bf Step 1:} Using the vector identity 
\begin{align*}
\eta^2 (\nabla \times \BE_h) \cdot (\nabla \times \BE_h) &= 
\nabla \times \BE_h \cdot \left( \nabla \times (\eta^2 \BE_h) - \nabla \eta^2 \times \BE_h\right) \\
& = (\nabla \times \BE_h) \cdot  \nabla \times (\eta^2 \BE_h)
- 2 \eta (\nabla \times \BE_h) \cdot (\nabla \eta \times \BE_h),
\end{align*}
we get 
\begin{align*}
& \norm{\nabla \times \BE_h}^2_{\BLT(B_R \cap \Omega)}  \leq 
\norm{\eta \nabla \times \BE_h}^2_{\BLT(\Omega)}  \\
& \qquad = 
\operatorname{Re} \left( 
a(\BE_h,\eta^2 \BE_h) + \kappa \langle\eta \BE_h,\eta \BE_h\rangle_{\BLT(B_R \cap \Omega)} 
-2\langle\eta \nabla \times \BE_h,\nabla \eta \times \BE_h\rangle_{\BLT(B_R \cap \Omega)}  \right)\\
& \qquad \leq \operatorname{Re} a(\BE_h,\eta^2 \BE_h) + \norm{\kappa}_{L^\infty} \norm{\BE_h}^2_{\BLT(B_{(1+\varepsilon)R}\cap \Omega)} 
+ 2 \norm{\eta \nabla \times \BE_h}_{\BLT(B_R \cap \Omega)} \norm{\nabla \eta \times \BE_h}_{\BLT(B_R \cap \Omega)}.
\end{align*}
Young's inequality then gives 
\begin{align}
\nonumber 
\norm{\nabla \times \BE_h}^2_{\BLT(B_R \cap \Omega)}  
&\leq \norm{\eta \nabla \times \BE_h}^2_{\BLT(\Omega)}  \\
\nonumber
& \leq  \operatorname{Re} a(\BE_h,\eta^2 \BE_h) + \norm{\kappa}_{L^\infty} \norm{\BE_h}^2_{\BLT(B_{(1+\varepsilon)R}\cap \Omega)}\\
\label{eq:caccioppoli-10a}
&\quad+ \frac{1}{2}\norm{\eta \nabla \times \BE_h}^2_{\BLT(B_R \cap \Omega)}+ 2 \norm{\nabla \eta \times \BE_h}^2_{\BLT(B_R \cap \Omega)} . 
\end{align}
Kicking  back the term $\frac{1}{2}\norm{\eta \nabla \times \BE_h}_{\BLT(B_R \cap \Omega)}^2$ to the left-hand side, we arrive at  
\begin{align}
\nonumber 
\norm{\nabla \times \BE_h}^2_{\BLT(B_R \cap \Omega)}  
&\leq \norm{\eta \nabla \times \BE_h}^2_{\BLT(\Omega)}  \\
\label{eq:caccioppoli-10}
& \leq 2 \operatorname{Re} a(\BE_h,\eta^2 \BE_h) + 2 (\norm{\kappa}_{L^\infty} + 2\norm{\nabla \eta}^2_{L^\infty}) \norm{\BE_h}^2_{\BLT(B_{(1+\varepsilon)R}\cap \Omega)}. 
\end{align}
Since 
$\norm{\kappa}_{L^\infty} + \norm{\nabla \eta}^2_{L^\infty} \lesssim (\varepsilon R)^{-2}$ with implied constant depending
on $\kappa$, we are left with estimating $\operatorname{Re} a(\eta \BE_h,\eta \BE_h)$. 

 {\bf Step 2:} 
 Using the orthogonality relation in the definition of the space $\mathcal{H}_{c,h}(B_{(1+\varepsilon)R}\cap \Omega)$, we get
 \begin{align}\label{proof of Cacciapoli. step1.3}
\operatorname{Re} a(\BE_h,\eta^2 \BE_h) &=  \operatorname{Re} a( \BE_h,   \eta^2 \BE_h- \Br_\Omega (\eta^2 \BE_h)) \nonumber  \\
&\lesssim 
\norm{\nabla \times \BE _h}_{\BLT (B_{(1+\varepsilon)R}\cap \Omega)}\norm{\nabla \times  \left( \eta^2 \BE_h- \Br_\Omega (\eta^2 \BE_h)\right) }_{\BLT (B_{(1+\varepsilon)R}\cap \Omega)} \nonumber\\
& \quad  +
\norm{\BE _h}_{\BLT (B_{(1+\varepsilon)R}\cap \Omega)}\norm{ \eta^2 \BE_h- \Br_\Omega (\eta^2 \BE_h) }_{\BLT (B_{(1+\varepsilon)R}\cap \Omega)}. 
 \end{align}
For each element $T \in \mathcal{T}_h$, Lemma~\ref{approximation properties of the local interpolation operator} yields
\begin{align}
\label{eq:caccioppoli-20}
\norm{\eta^2 \BE_h- \Br_\Omega (\eta^2 \BE_h)}^2_{\BLT (T) }  + 
\norm{\nabla \times  \left( \eta^2 \BE_h- \Br_\Omega (\eta^2 \BE_h)\right) }^2_{\BLT (T) } 
\lesssim h^2 \left(
|\eta^2 \BE_h|^2_{\BH^1(T)} + 
|\nabla \times (\eta^2 \BE_h)|^2_{\BH^1(T)} 
\right). 
\end{align}
To proceed further, we observe that $\BE_h|_T \in {\mathcal N}_0(T)$ has the form $\BE_h = {\mathbf a} + {\mathbf b} \times {\mathbf x}$ so that 
$\operatorname{curl} \BE_h|_T  = 2 {\mathbf b}$ and hence 
$\sum_{j= 1}^3 |\partial_{x_j}  \BE_h| \lesssim |\nabla \times \BE_h|$ pointwise on $T$ so that we get with an implied constant independent of the function
$\eta$ 
\begin{equation}
\label{eq:caccioppoli-30}
\sum_{j=1}^3 \norm{\eta \partial_{x_j} \BE_h}_{\BLT(T)} \lesssim \norm {\eta \nabla \times \BE_h}_{\BLT(T)}. 
\end{equation}
Using \eqref{eq:caccioppoli-30} we obtain 
\begin{align}
\label{eq:caccioppoli-40}
|\eta^2 \BE_h|_{\BH^1(T)} &\lesssim \frac{1}{\varepsilon R} \norm{\BE_h}_{\BLT(T)} + \|\eta \nabla \times \BE_h\|_{\BLT(T)}.  
\end{align}
Computing $\nabla \times (\eta^2 \BE_h) = \nabla \eta^2 \times \BE_h + \eta^2 \nabla \times \BE_h$, 
using the product rule and the fact that $\partial_{x_j} (\nabla \times \BE_h) = 0$ since $\nabla \times \BE_h$ is constant 
gives again in view of \eqref{eq:caccioppoli-30} and $\varepsilon R \lesssim 1$
\begin{align}
\label{eq:caccioppoli-50}
|\nabla \times ( \eta^2 \BE_h)|_{\BH^1(T)} &\lesssim \frac{1}{(\varepsilon R)^2} \norm{\BE_h}_{\BLT(T)} 
+ \frac{1}{\varepsilon R} \norm{\eta \nabla \times \BE_h}_{\BLT(T)}.
\end{align}
Summing the squares of \eqref{eq:caccioppoli-40}, \eqref{eq:caccioppoli-50} over all 
elements $T$ with $T \cap \operatorname{supp} \eta \ne \emptyset$, which is ensured 
if we sum over all $T$ with $T \subset B_{(1+\varepsilon)R} \cap \Omega$, and inserting the result 
in \eqref{eq:caccioppoli-20} yields 
\begin{align*}
  \operatorname{Re} a( \BE_h,   \eta^2 \BE_h- \Br_\Omega (\eta^2 \BE_h))  
 &\lesssim  
\left(\norm{\nabla \times \BE_h}_{\BLT(B_{(1+\varepsilon)R} \cap \Omega )}
+ \norm{\BE_h}_{\BLT(B_{(1+\varepsilon)R} \cap \Omega)}
\right) \times \\
& \qquad \frac{h}{\varepsilon R} \left( \frac{1}{\varepsilon R}\norm{\BE_h}_{\BLT(B_{(1+\varepsilon)R}\cap \Omega)} + \norm{\eta (\nabla \times \BE_h)}_{\BLT(B_{(1+\varepsilon)R}\cap \Omega)} \right).
\end{align*}
Using  Young's inequality, $h \lesssim 1$ and $0\leq \eta \leq 1$ as well as the definition of the norm $\vertiii{\cdot} _{c,h, R}$, we obtain 
\begin{align*}
  \operatorname{Re} a( \BE_h,   \eta^2 \BE_h- \Br_\Omega (\eta^2 \BE_h))  &\lesssim \frac{h^2}{(\varepsilon R)^2}  \norm{\nabla \times \BE_h}^2_{\BLT(B_{(1+\varepsilon)R} \cap \Omega )} +
\frac{1}{(\varepsilon R)^2} \norm{\BE_h}^2_{\BLT(B_{(1+\varepsilon)R}\cap \Omega)}  \\
& \quad+  
 \frac{h}{\varepsilon R} \norm{\nabla \times \BE_h}_{\BLT(B_{(1+\varepsilon)R} \cap \Omega )} \norm{\eta \nabla \times \BE_h}_{\BLT(B_{(1+\varepsilon)R}\cap \Omega)}
  \\& 
\lesssim \varepsilon^{-2} \vertiii{\BE_h}_{c,h,(1+\varepsilon)R}^2 + 
\varepsilon^{-1} \vertiii{\BE_h}_{c,h,(1+\varepsilon)R} \norm{\eta \nabla \times \BE_h}_{\BLT(B_{(1+\varepsilon)R}\cap \Omega)}. 
\end{align*}
Inserting this in \eqref{eq:caccioppoli-10} produces 
\begin{align*}
\norm{\nabla \times \BE_h}^2_{\BLT(B_R \cap \Omega)} & \leq 
\norm{\eta \nabla \times \BE_h}^2_{\BLT(\Omega)} \\
& \lesssim 
\varepsilon^{-2} \vertiii{\BE_h}_{c,h,(1+\varepsilon)R}^2 + 
\varepsilon^{-1} \vertiii{\BE_h}_{c,h,(1+\varepsilon)R} \norm{\eta \nabla \times \BE_h}_{\BLT(B_{(1+\varepsilon)R}\cap \Omega)}. 
\end{align*}
Using again Young's inequality to kick the term $\norm{\eta \nabla \times \BE_h}_{\BLT(B_{(1+\varepsilon)R}\cap \Omega)}$ of the right-hand side
back to the left-hand side produces the desired estimate. 
\qed
 \end{proof}

For functions in $\CH _{g,h} (B _{(1+\varepsilon)R}\cap \Omega)$, a discrete Caccioppoli-type estimate has already been established in \cite[Lem.~{2}]{faustmann2015mathcal}, which we state in the following for sake of completeness.

 \begin{lemma} [\protect{\cite[Lem.~{2}]{faustmann2015mathcal}}]
\label{th:Caccioppoli-irrotational part}
 	 	Let   $\varepsilon \in (0,1)$ and $R \in (0,2\operatorname*{diam}(\Omega))$ be  such that  $\frac{h}{R} < \frac{\varepsilon}{4}$. 
 	 Let 
 	$B_R$ and $B_{(1+\varepsilon)R}$ be  two concentric boxes and $p _h \in \mathcal{H}_{g,h}(B_{(1+\varepsilon)R}\cap \Omega)$.
 	Then, there exists a constant $C >0$ depending only on $\Omega$ 
 	and the $\gamma$-shape regularity of the quasi-uniform triangulation $\T_h$ such that 
 	\begin{align}
 	\nonumber
 	\left\| \nabla p _h \right\|_ {\BLT(B_{R}\cap \Omega)}
 	&\le 
 	C \frac{1+\varepsilon}{\varepsilon} \vertiii{p _h}_{g,h,(1+\varepsilon) R} . 
 	\end{align}
 \end{lemma}

%-----------------------------------------------------------------------------------------------------------------------------------
\subsection{Low-dimensional approximation in $\mathcal{H}_{c,h}(B_R \cap \Omega)$. }
In this subsection, we apply  the Caccioppoli-type estimates from Lemmas~\ref{th:Caccioppoli-divergence free part} 
and  \ref{th:Caccioppoli-irrotational part} to find approximations of the Galerkin solutions from low-dimensional spaces. 
We will need a Poincar\'e inequality as given in \cite[(7.45)]{gilbarg-trudinger77a}: for open sets $D\subset \omega$ with $|D| > 0$ and $u \in H^1 (\omega)$, we have
\begin{align}
\label{eq:poincare-GT}
\norm{u- \frac{1}{\abs{D}} \int_{D}u \,dx}_{L ^2(\omega)}\lesssim {\abs{D}}^ {-2/3}(\diam(D))^3 \norm{\nabla u}_{L ^2(\omega)}. 
\end{align}
In the following, we consider  low-dimensional approximation of discrete harmonic functions in Lemma~\ref{H-Matrix ApproximationL-FMP15} that generalizes 
\cite[Lem.~4]{faustmann2015mathcal}. 
    \begin{lemma}
    	\label{H-Matrix ApproximationL-FMP15}

    	Let $ \varepsilon \in ( 0,1) $, $q \in (0,1) $, $R \in (0,2\operatorname*{diam}(\Omega))$, and $m \in \mathbb{N}$ satisfy 
    	\begin{align}
    	\label{eq:proofHmatrixlemma1}
    	 	\frac{h}{R} \leq \frac{ q  \varepsilon  }{8m \max \left\lbrace 1 , C   _{\rm app} \right\rbrace }, 
    	\end{align}
    	where  the constant $C _{\rm app}$ is given in \cite[Lem.~3, Lem.~4]{faustmann2015mathcal} and depends only on $\Omega$ 
        and the $\gamma$-shape regularity of the quasi-uniform triangulation $\T_h$. 
    	Let $B_R$, $B_{(1+ \varepsilon )R}$, $B_{(1+2\varepsilon)R}$ be concentric boxes. 
    	Then, there exists a subspace $W_{m} $ of   $ \mathcal{H}_{g,h}( B _{R} \cap \Omega)$ 
    	of dimension 
    	\begin{align*}
    	\dim W_{m} \leq C ^\prime  _{\dim}\left(\frac{1+  \varepsilon  ^ {-1}}{q} \right)^{3}  {m }^4
    	\end{align*}
    	with the following  approximation properties:
\begin{enumerate}[(i)]
\item  If $u _h \in \CH _{g,h}(B _{(1+\varepsilon)R}\cap \Omega)$ and $\overline{B_{(1+\varepsilon)R}} \cap \Omega^c = \emptyset$, then
     	\begin{align*}
   	 \min\limits_{ \widetilde{ u}  _{m} \in W _{m}}
	\vertiii{u _h - \widetilde{u} _{m}}_{g,h,R}
	&\leq C_{\rm app}^\prime  {q  }  ^{m} \norm{\nabla{u _h}}_{\BLT(B _{(1+\varepsilon) R}\cap \Omega)}. 
\end{align*}
\item  If $u _h \in \CH _{g,h}(B _{(1+2\varepsilon)R}\cap \Omega)$ and $\overline{B_{(1+\varepsilon)R}} \cap \Omega^c \ne \emptyset$, then
\begin{align*}
    	 \min\limits_{ \widetilde{ u}  _{m} \in W _{m}}
    	\vertiii{u _h - \widetilde{u} _{m}}_{g,h,R}
    	&\leq C_{\rm app} ^\prime {q  }  ^{m} \varepsilon^{-2} \norm{\nabla{u _h}}_{\BLT(B _{(1+2\varepsilon) R}\cap \Omega)}.  
    	\end{align*}
\end{enumerate}
    	Here, $ C ^\prime  _{\dim}$, $C_{\rm app}^\prime$ depend  only on 
    	$\Omega$  and the $\gamma$-shape regularity of the quasi-uniform triangulation $\T_h$.
    \end{lemma}
\begin{proof}
We start with the case of boxes not entirely contained in $\Omega$. 
\newline 
\textbf{Case 1:} Let $\overline{B_{(1+\varepsilon)R}}  \cap \Omega ^c \neq \emptyset$. 
For the Lipschitz domain $\Omega$, \cite[Chap.~{VI}, Sec.~{3}, Thm.~{5'}]{stein1970singular} asserts the existence of a bounded linear extension operator 
$\CE_{\Omega^c}: H ^1 (\Omega ^c) \rightarrow H^1 (\R ^3)$ such that $\CE_{\Omega^c} v | _ {\Omega ^c} = v$ 
for each $v \in H ^1 (\Omega ^c) $. The fact that $\Omega^c$ is Lipschitz 
(see \cite[Thm.~{2}]{hajlasz2008sobolev} for details) 
implies the existence of a 
constant $c >0$ depending only on $\Omega$ such that for all $x \in \Omega^c$ and all $r \in (0,1)$ 
we have $|B_{r}(x) \cap \Omega^c| \ge c r^3$, where $B_r(x)$ denotes the ball of radius $r$ centered at $x$. 
Selecting an $x \in B_{(1+\varepsilon)R} \cap \Omega^c$ and noting that 
$B_{\varepsilon R/2}(x) \subset B_{(1+2\varepsilon)R}$, we conclude 
$$
|B_{(1+2\varepsilon)R} \cap \Omega^c| \ge 
|B_{\varepsilon R/2}(x) \cap \Omega^c| \ge c (\varepsilon R/2)^{3} . 
$$
Due to \eqref{eq:proofHmatrixlemma1}, \cite[Lem.~4]{faustmann2015mathcal} provides a subspace $ W_{m} $ of   $ \mathcal{H}_{g,h}( B _{R} \cap \Omega)$ such that
	\begin{align}\label{FMP15}
 	 \min\limits_{  \widetilde  { u}  _{m} \in  W _{m}}
 	\vertiii{ u _h - \widetilde  {u} _{m}}_{g,h,R}
 	&\leq  {q  }  ^{m}	\vertiii{{  u _h}}_{g,h,(1+\varepsilon)R},  \\
 \label{72a}
\dim  W_{m} & \leq  C   _{\dim}\left(\frac{1+  \varepsilon  ^ {-1}}{q} \right)^{3}  {m }^4,
\end{align}
where  $  C _{\dim}$ depends  only on $\Omega$  and the $\gamma$-shape regularity of the quasi-uniform triangulation $\T_h$.
We denote by $\widehat u _ h $ the extension by zero of $u_h$ to $\Omega ^c$.
It follows from the Poincar\'e inequality 
(\ref{eq:poincare-GT}) 
and   $\abs{B _ {(1+2\varepsilon)R} \cap \Omega ^c } \gtrsim (\varepsilon R)^3$ that 
\begin{align}
	\nonumber
	\frac{1}{ R}\norm{ u _h}_{L ^2(B _{(1+\varepsilon) R} \cap \Omega)}
	&  \le \frac{1}{ R}\norm{ u _h}_{L ^2(B _{(1+2\varepsilon) R} \cap \Omega)}
	=
	\frac{1}{ R}\norm{\widehat u _h}_{L ^2(B _{(1+2\varepsilon) R})}
	\\ &
	\nonumber  \lesssim \frac{\abs{B _{(1+2\varepsilon) R}}}{ R \abs{B _ {(1+2\varepsilon)R} \cap \Omega ^c }^ {2/3} } \norm{ \nabla \widehat u _h}_{\BL  ^2(B _{(1+2\varepsilon) R})} \\
	&\lesssim \frac{(1+2\varepsilon )^3 R^3}{\varepsilon ^ 2R^3} \norm{ \nabla \widehat u _h}_{\BLT(B _{(1+2\varepsilon) R})}
\lesssim \varepsilon ^ {-2}\norm{ \nabla \widehat u _h}_{\BL  ^2(B _{(1+2\varepsilon) R})}.
\label{Poincar\'e inequality}
\end{align}
Combining \eqref{Poincar\'e inequality} and \eqref{FMP15} leads to 
\begin{align}\label{FMP15a}
	& \min\limits_{   \widetilde{ u}  _{m}  \in  W _{m}}
	\vertiii{ u _h -  \widetilde{ u}  _{m} }_{g,h,R} \lesssim
	\varepsilon ^ {-2} q ^m\norm{ \nabla  u _h}_{\BLT(B _{(1+2\varepsilon) R} \cap \Omega)}.  
\end{align}
\textbf{Case 2:}
Let $\overline{B_{(1+\varepsilon)R}} \cap \Omega^c = \emptyset$.  We note that constant functions are in $ \mathcal{H}_{g,h}(B_R \cap \Omega)$. Hence, 
by \cite[Lem.~4]{faustmann2015mathcal} there is a subspace $ W_{m} \subset  \mathcal{H}_{g,h}( B _{R} \cap \Omega)$ such that $1 \in W_m$ and 
\begin{align}\label{FMP15b}
& \min\limits_{  \widetilde{u} _{m} \in  W _{m}}
\vertiii{ u _h - \widetilde{u} _{m}}_{g,h,R}
=\min\limits_{ \widetilde{u} _{m} \in  W _{m},\, c \in \R}
\vertiii{ u _h -  \widetilde{u} _{m}+c}_{g,h,R} \leq  {q  }  ^{m}	\min _ {c \in \R}\vertiii{{  u _h-c}}_{g,h,(1+\varepsilon)R} 
\end{align}
	with dimension 
\begin{align*}
\dim W_{m} \leq C   _{\dim}\left(\frac{1+  \varepsilon  ^ {-1}}{q} \right)^{3}  {m }^4 + 1 
\lesssim \left(\frac{1+  \varepsilon  ^ {-1}}{q} \right)^{3}  {m }^4.  
\end{align*}
A standard Poincar\'e inequality (i.e.,  (\ref{eq:poincare-GT}) with $D = B_{(1+\varepsilon)R}$) implies
\begin{align}\label{FMP15c}
\nonumber \min _ {c \in \R} \vertiii{{  u _h-c}}_{g,h,(1+\varepsilon)R} &\le \vertiii{{  u _h-\frac{1}{\abs{{B _{(1+\varepsilon)R}}}} \int_{B _{(1+\varepsilon)R}}u _h}}_{g,h,(1+\varepsilon)R} \\
\nonumber & \lesssim \frac{\abs{B _{(1+\varepsilon) R}}}{ R \abs{B _ {(1+\varepsilon)R} }^ {2/3} }\norm{ \nabla  u _h}_{\BLT(B _{(1+\varepsilon) R}\cap \Omega)}+ \frac{h}{(1+\varepsilon)R} \norm{ \nabla u _h}_{\BLT(B _{(1+\varepsilon) R}\cap \Omega)}\\
& \lesssim  \norm{ \nabla  u _h}_{\BLT(B _{(1+\varepsilon) R}\cap \Omega)}.
\end{align}
Combining \eqref{FMP15c} and \eqref{FMP15b} completes the proof.
\qed
\end{proof}
\begin{remark}
The factor $\varepsilon^{-2}$ instead of $\varepsilon^{-0}$ for boxes $B_R$ near the boundary is a consequence of not assuming a relation between
the orientation of the boxes and the boundary. Aligning boxes with the boundary allows one to better exploit boundary conditions
and improve the factor $\varepsilon^{-2}$. 
\eremk
\end{remark}
In the following, we will need a simplified version of Lemma~\ref{H-Matrix ApproximationL-FMP15}: 
\begin{corollary}
\label{cor:H-Matrix ApproximationL-FMP15} 
Let $R \in (0,2\operatorname*{diam}(\Omega))$, $\varepsilon \in (0,1)$, $q \in (0,1)$. There are constants
$C^{\prime\prime}_{\rm dim}$ and $C^{\prime\prime}_{\rm app}$ depending only on $\Omega$ and the $\gamma$-shape regularity of the quasiuniform triangulation $\T_h$
such that, for any concentric boxes $B_R$, $B_{(1+2\varepsilon)R}$ and any $m \in {\mathbb N}$, there exists a subspace $W_m \subset 
\CH_{g,h}(B_R\cap\Omega)$ of dimension 
\begin{align*}
\operatorname{dim} W_m \leq C^{\prime\prime}_{\rm dim} (\varepsilon q)^{-3} m^4
\end{align*} 
such that for any $u_h \in \CH_{g,h}(B_{(1+2\varepsilon)R}\cap\Omega)$ there holds 
\begin{align}
\label{eq:cor:H-Matrix ApproximationL-FMP15-10}
\min_{\widetilde u_m \in W_m} \vertiii{u_h - \widetilde u_m}_{g,h,R} \leq C^{\prime\prime}_{\rm app} q^m \varepsilon^{-2} \norm{\nabla u_h}_{\BLT(B_{(1+2\varepsilon)R}\cap\Omega)}. 
\end{align}
\end{corollary}
\begin{proof}
The case that the parameters satisfy (\ref{eq:proofHmatrixlemma1}) is covered by Lemma~\ref{H-Matrix ApproximationL-FMP15}. For the converse
case $h/R > q \varepsilon/(8 m \max\{1,C_{\rm app}\})$, we take $W_m:= \CH_{g,h}(B_R \cap \Omega)$ so that the minimum 
in (\ref{eq:cor:H-Matrix ApproximationL-FMP15-10}) is zero and observe in view of the quasi-uniformity of $\T_h$
\begin{align*}
\operatorname{dim} \CH_{g,h}(B_R \cap \Omega) \lesssim \left(\frac{R}{h}\right)^3 
\lesssim \left(\frac{m}{\varepsilon q}\right)^3 =  (\varepsilon q)^{-3} m^3 \leq (\varepsilon q)^{-3} m^4,
\end{align*}
which finishes the proof.
\qed
\end{proof}
If $\BE_h$ is locally discrete divergence-free, then the function $\nabla (p + \varphi_{\Bz})$ in the decomposition 
$\BE_h = \Bz  - \nabla \varphi_{\Bz} + \nabla (p+ \varphi_{\Bz})$ given by Definition~\ref{def:local-helmholtz} is also 
locally discrete divergence-free since $\Bz - \nabla \varphi_{\Bz}$ is divergence-free. 
The following lemma shows that $\Pi^\nabla_{\widetilde B_{(1+2\varepsilon)R}} \nabla (p + \varphi_{\Bz})$ is discrete divergence-free as well: 
     \begin{lemma}\label{H-Matrix ApproximationL4a}
    	Let    $\varepsilon \in (0,1)$, $R \in (0,2\operatorname*{diam}(\Omega))$, and let $B_ {(1+j\varepsilon)R}$, $j\in\{0,1,2\}$, be  concentric boxes.
    Introduce 	$\T _h(B_{(1+2\varepsilon)R} \cap \Omega)$ and $\widetilde B _{(1+2\varepsilon)R}$ according to Assumption~\ref{def:mesh-conforming}. 
 Let 
    	$\eta \in C ^\infty (\overline{\Omega}) $
    	 be a  cut-off function with  $\eta \equiv 1$ on $\widetilde B_{(1+2\varepsilon)R}$. Let $\BE_h$ be such that 
$\eta \BE_h \in \BH_0(\operatorname{curl},\Omega)$ and 
$\BE _h \in \mathcal{H} _{c,h}(B_ {(1+2\varepsilon)R}\cap \Omega)$. 
Decompose    $\eta \BE _h \in \BH _0(\operatorname*{curl},\Omega)$  as
$\eta \BE _h =  \Bz + \nabla p$
with $ \Bz \in \BH _0  ^1 (\Omega)$ and $p \in H_0^1 (\Omega)$  according to Lemma \ref{Helmholtz decomposition}.
Let   the mapping  $\varphi _ {\Bz} : \BH _0^1(\Omega) \rightarrow H^1 _0 (\Omega)$ be  defined according to  \eqref{eq:varphi_z} 
taking $\widetilde{ \eta} \equiv \eta $ there. 
    	 Then, 
    $\Pi^\nabla_{\widetilde{B}_{(1+2\varepsilon) R}} \nabla (p + \varphi_{\Bz})$ is discrete divergence-free 
    	on $\widetilde{B}_{(1+2\varepsilon) R}$, i.e., 
    	\begin{equation}
    	\label{eq:ph+varphi_z-disc-harmonic}
    	\langle \Pi^\nabla_{\widetilde{B}_{(1+2\varepsilon) R}} \nabla (p + \varphi_{\Bz}), \nabla v_h\rangle_{\BLT(\widetilde{B}_{(1+2\varepsilon) R})}
    	=  0 \qquad \forall v_h \in S^{1,1}(\T_h,\widetilde{B}_{(1+2\varepsilon) R}), 
\quad \operatorname{supp} v_h \subset \overline {\widetilde B_{(1+2\varepsilon) R}}. 
    	\end{equation}
    \end{lemma}
\begin{proof}
We use $\BE _h \in \mathcal{H} _{c,h}(B_ {(1+2\varepsilon)R} \cap \Omega)$ and \eqref{eq:local-div} so that, 
for $v_h \in S^{1,1}(\T_h,\widetilde{B}_{ (1+2\varepsilon) R})$ with $\operatorname{supp} v_h \subset \overline{\widetilde{B}_{(1+2\varepsilon) R}}$, we have 
\begin{align*}
0 & = a (\BE_h,\nabla v_h)=
	\skp{\nabla \times \BE_h, \nabla \times \nabla v_h}_{\BLT(\widetilde{B}_{(1+2\varepsilon) R})}   -\kappa   \skp{\BE_h, \nabla v_h}_{\BLT(\widetilde{B}_{(1+2\varepsilon) R})}  
	\\&=-\kappa \langle\BE_h,\nabla v_h\rangle_{\BLT(\widetilde{B}_{(1+2\varepsilon) R})}  =-\kappa \langle\eta\BE_h,\nabla v_h\rangle_{\BLT(\widetilde{B}_{(1+2\varepsilon) R})}
	=- \kappa \langle\Bz + \nabla p,\nabla v_h\rangle_{\BLT(\widetilde{B}_{(1+2\varepsilon)R})}  
	\\&=-\kappa  \langle\Bz -\nabla \varphi_{\Bz} + \nabla \varphi_{\Bz} + \nabla p,\nabla v_h\rangle_{\BLT(\widetilde{B}_{(1+2\varepsilon) R})}   \\
	&=-\kappa  \langle(\Bz -\nabla \varphi_{\Bz}) + \Pi^\nabla_{\widetilde{B}_{(1+2\varepsilon)R}} (\nabla \varphi_{\Bz} + \nabla p),\nabla v_h\rangle_{\BLT(\widetilde{B}_{ (1+2\varepsilon) R})}    \\
	& 
	\stackrel{\text{Lem.~\ref{lemma:varphi_z}}}{=} -\kappa \langle \Pi^\nabla_{\widetilde{B}_{ (1+2\varepsilon)R}} (\nabla \varphi_{\Bz} + \nabla p),\nabla v_h\rangle_{\BLT(\widetilde{B}_{(1+2\varepsilon) R})},  
\end{align*}
which finishes the proof.
\qed
\end{proof}
We will make use of the orthogonal projection
\begin{equation}
\label{eq:Pih}
\boldsymbol \Pi _{B_R}  \colon (\BH(\operatorname*{curl},B_R \cap \Omega) , \vertiii{\cdot}_{c,h,R})\rightarrow (\mathcal{H} _{c,h}(B_R \cap \Omega), \vertiii{\cdot}_{c,h,R}),
\end{equation}
where orthogonality is defined in terms of the inner product associated with the weighted norm $\vertiii{\cdot}_{c,h,R}$.
   \begin{lemma} [single-step approximation]
   	\label{H-Matrix ApproximationL4}
   	Let    $\varepsilon \in (0,1)$, $R>0$ be such that 
   	$(1+4\varepsilon)R \in (0,R _ {\rm max}]$, and $q \in (0,1)$. 
        Let $B_ {(1+j\varepsilon)R}$, $j=0,\ldots, 4$, be  concentric boxes.  
Then, there exists a family of linear spaces $\BV_{H,m} \subset \mathcal{H}_{c,h}(B_{R} \cap \Omega)$ (parametrized by $H  > 0$, $m \in {\mathbb N}$) 
with the following approximation properties: For each $\BE _h \in \mathcal{H} _{c,h}(B_ {(1+4\varepsilon)R} \cap \Omega)$ there is a 
   	   $\BE _{1,h}\in \BV_{H,m} \subset  \mathcal{H} _{c,h}(B_ {R} \cap \Omega) $ with 
   \begin{enumerate}[(i)]
   	\item
\label{item:H-Matrix ApproximationL4-i}
  $(\BE _h - \BE _ {1,h} )|_ {B_ R \cap \Omega} \in \mathcal{H} _ {c,h}(B_R \cap \Omega), $
  \item
\label{item:H-Matrix ApproximationL4-ii}
 $\displaystyle \vertiii{\BE _h-\BE _{1,h}} _ {c,h,R}  
 \leq  C^{\prime\prime}_{\rm app} 
 \left(\frac{H}{R}\varepsilon^{-1} +q^m \varepsilon^{-3}\right)  \vertiii{\BE_h}_{c,h,(1+4\varepsilon)R},
 $
 \item
\label{item:H-Matrix ApproximationL4-iii}
$  \displaystyle \dim \BV_{H,m} \le
  C_{\rm dim}^{\prime\prime}\left[ \Bigl( \frac{R}{H}\Bigr)  ^{3} + \bigl(\varepsilon q \bigr)^{-3}{m} ^4\right]$,\\
\end{enumerate}
   	   where   
   	the constants $C_{\rm app}^{\prime\prime}$ and  $C_ {\rm dim}^{\prime\prime}$ depend only on $\kappa$, $\Omega$, and the $\gamma$-shape regularity of the quasi-uniform
triangulation $\T_h$. 
Furthermore, 
\begin{enumerate}[(i)]
\setcounter{enumi}{3}
	\item
\label{item:H-Matrix ApproximationL4-iv}
 if $h \ge H$ or $h/R \ge \varepsilon/4$, one may actually take $\BV_{H,m} = \CH_{c,h}(B_R\cap\Omega)$
and $\BE_{1,h}$ may be taken as $\BE_{1,h} = \BE_h|_{B_R\cap\Omega}$. 
\end{enumerate}
   \end{lemma}
   \begin{proof} 
\textbf{Step 1:} (reduction to $h < H$)
As a preliminary step, we show (\ref{item:H-Matrix ApproximationL4-iv}) so that afterwards we 
may restrict our attention to the case $h <  H$ together with $h/R < \varepsilon/4$. 
If $h \ge H$ or $h/R \ge \varepsilon/4$, we take 
$\BV_{H,m}:= \CH_{c,h}(B_R \cap \Omega)$, which implies that the choice $\BE_{1,h} = \BE_h|_{B_R\cap\Omega} $ is admissible
so that $\BE_h - \BE_{1,h} = 0$. Since either $h \ge  H$ or $ h/R \ge \varepsilon/4$, we have 
\begin{align}
\label{eq:H-Matrix ApproximationL4-space dimension}
\operatorname{dim} \CH_{c,h}(B_R\cap\Omega) \lesssim \left(\frac{R}{h}\right)^3 
\lesssim \left(\frac{R}{H}\right)^3 +  \varepsilon^{-3} 
\lesssim \left(\frac{R}{H}\right)^3 +  (\varepsilon q)^{-3},  
\end{align}
which shows that the complexity bound in (\ref{item:H-Matrix ApproximationL4-iii}) is satisfied. 
We have thus shown (\ref{item:H-Matrix ApproximationL4-iv}) and will assume $h <H$ and $h/R <  \varepsilon /4$ for the remainder of the proof. 

\textbf{Step 2:} (reduction to $H/R \le \varepsilon/4$)
 For $\frac{H}{R} > \frac{\varepsilon}{4}$, we may take the space constructed below with 
the choice $\frac{H}{R} = \frac{\varepsilon}{4}$ since then, the approximation property 
(\ref{item:H-Matrix ApproximationL4-ii}) and the complexity estimate 
(\ref{item:H-Matrix ApproximationL4-iii}) are still satisfied. 
Therefore, we assume in the remainder that 
 $\frac{H}{R} \le \frac{\varepsilon}{4}$.

\textbf{Step 3:} (Scott-Zhang approximation on ${\mathbb R}^3$)
 Let $\mathcal{M}_H$ be a quasi-uniform infinite triangulation of $\mathbb{R} ^3$ with mesh width $H$. Define further
$\displaystyle \BS^{1,1}(\mathcal{M}_H): = \lbrace  \Bp_H \in \BH^1 (\mathbb{R}^3 ) 
\;\colon \;\Bp_H | _M \in ( \mathcal{P}_1(M))^3 \quad \forall M \in \mathcal{M}_H \rbrace.$
We will use the Scott–Zhang projection  operator
$\BI_H ^{\rm {SZ}}\colon\BH ^1(\mathbb{R} ^3) \rightarrow \BS^{1,1}(\mathcal{M}_H)$ 
introduced in \cite{scott1990finite}. 
Denoting by $\omega _ M$ the element patch of $M \in \mathcal{M} _h $, this operator has the local approximation property 
\begin{align}
\label{eq:SZ-approximation-a}
\norm{\BU - \BI_H ^{\rm {SZ}} \BU}^2_ { {\BLT (M)}} &\le CH^2 \norm{\BU} ^2_ {\BH ^1 (\omega _M)} 
\qquad \forall \BU \in \BH ^1 (\omega _M)
\end{align}
with a  constant $C$ depending only on $\Omega$ 
and the $\gamma$-shape regularity of the quasi-uniform triangulation $\mathcal{ M}_H$.
   	Let $ \boldsymbol{\CE}  \colon \BH ^1 (\Omega) \rightarrow \BH^1 (\R^3)$ 
        be an $\BH ^1$-stable 
        extension operator such as the one from \cite[Chap.~{VI}, Sec.~{3}, Thm.~{5'}]{stein1970singular}. 

\textbf{Step 4:} 
   	Let 	$\T _h(B_{(1+2\varepsilon)R} \cap \Omega)$ and $\widetilde B _{(1+2\varepsilon)R }$ 
be given according to Assumption~\ref{def:mesh-conforming}. 
   	 Let 
   	$\eta \in C^\infty (\overline{\Omega}) $ be a  cut-off function with $\operatorname{supp} {\eta} \subseteq \overline{B _ {(1+3\varepsilon)R}\cap \Omega}$,  $\eta \equiv 1$ on $\widetilde B_{(1+2\varepsilon)R}$, $0 \le \eta \le 1$ and  $ \left\| \nabla ^ \ell \eta  \right\|_{L^\infty(\Omega)} 
   	\lesssim \frac{1}{ (\varepsilon R)^ \ell}$ for $\ell  \in \{0,1,2\}$. 
Note that $\eta \BE_h \in \BH_0(\operatorname{curl},\Omega)$. 
Decompose $\eta \BE _h \in \BH _0(\operatorname*{curl},\Omega)$ as $\eta \BE _h = \Bz + \nabla p$ 
   	with $ \Bz \in \BH _0  ^1 (\Omega)$ and $p \in H_0^1 (\Omega)$  according to Lemma~\ref{Helmholtz decomposition}. 
Let $\varphi_{\Bz}$ be given by (\ref{eq:varphi_z}) taking $\widetilde\eta = \eta$ there. 
Select representers $p_h$, $\varphi_{\Bz,h}
\in S^{1,1}_0(\T_h)$ such that $\nabla p_h = \Pi^\nabla_{\widetilde{B}_{(1+2\varepsilon) R}}  \nabla p$ 
   	and $\nabla \varphi_{\Bz,h} =  \Pi^\nabla_{\widetilde{B}_{(1+2\varepsilon) R}} \nabla \varphi_{\Bz}$ on $\widetilde{B}_{(1+2\varepsilon) R}$. 
By Lemma~\ref{H-Matrix ApproximationL4a}, we have that $\nabla (p_h + \varphi_{\Bz,h})$ is discrete divergence-free on $\widetilde B_{(1+2\varepsilon)R}$ 
so that $ (p_h + \varphi_{\Bz,h}) \in \mathcal{H}_{g,h}{(B_{(1+2\varepsilon) R} \cap \Omega)}$. 
We apply Corollary~\ref{cor:H-Matrix ApproximationL-FMP15}  with the pair $(R,\varepsilon)$ replaced 
with $(\widetilde R,\widetilde \varepsilon) = (R(1+\varepsilon),\frac{\varepsilon}{2(1+\varepsilon)})$ to get a 
subspace $W _{m} \subset \mathcal{H}_{g,h}{(B_{(1+\varepsilon)R} \cap \Omega)}$  
for the box $B_{(1+\varepsilon)R} \cap \Omega$ and an $w_m \in W_m$ such that 
\begin{equation}
\label{eq:approx-ph+phih}
\vertiii{p_h + \varphi_{\Bz,h}-w_m}_{g,h,(1+\varepsilon) R} \lesssim 
q^m \varepsilon^{-2} \|\nabla (p_h + \varphi_{\Bz,h})\|_{\BLT(B_{(1+2\varepsilon)R} \cap \Omega)}.
\end{equation}
\textbf{Step 5:} 
   		Define 
   	$\Bz_ {H} :=(\BI^{SZ}_H \boldsymbol \CE\Bz)|_{B_{(1+4\varepsilon)R}\cap \Omega}$. 
Using Definition~\ref{The local  Helmholtz decomposition} 
and with the function $\varphi_{\Bz_H}$ given by (\ref{eq:varphi_z}) (again, with $\widetilde\eta = \eta$ there) 
we have 
 the representation 
\begin{align*}
\BE_h | _{\widetilde{B}_{(1+2\varepsilon) R}}& = \Bz_h + \Pi^\nabla_{\widetilde{B}_{(1+2\varepsilon) R}}  \nabla p = (\Bz_h-\Bz)+\Bz - \Pi^\nabla_{\widetilde{B}_{(1+2\varepsilon) R}} \nabla \varphi_{\Bz} +\Pi^\nabla_{\widetilde{B}_{(1+2\varepsilon) R}} \nabla(\varphi_{\Bz} +  p)\\
     & = (\Bz_h-\Bz)+(\Bz -\Bz_H) - \Pi^\nabla_{\widetilde{B}_{(1+2\varepsilon) R}} (\nabla \varphi_{\Bz}  - \nabla \varphi_{\Bz_H})  \\
& \quad \mbox{} 
- \Pi^\nabla_{\widetilde{B}_{(1+2\varepsilon) R}} \nabla \varphi_{\Bz_H} 
+\Bz_H + \Pi^\nabla_{\widetilde{B}_{(1+2\varepsilon) R}} \nabla(\varphi_{\Bz} + p). 
\end{align*}
Of these 6 terms, the first three terms are shown to be small, the next two terms are from 
a low-dimensional space, and the last term is exponentially (in $m$) close to $\nabla w_m$ by (\ref{eq:approx-ph+phih}), which is 
also from a low-dimensional space, namely,  $\nabla W_m$. 
As the approximation of $\BE_h$, we thus take 
   	\begin{equation}
\label{eq:Eh1}
\BE _{1,h}:= \boldsymbol \Pi _{B_R} \left(  - \Pi^\nabla_{\widetilde{B}_{(1+2\varepsilon) R}} \nabla \varphi_{\Bz_H} +\Bz_H +\nabla w_m \right), 
\end{equation}
with the $\vertiii{\cdot}_{c,h,R}$-orthogonal projection $\boldsymbol \Pi_{B_R}$ of (\ref{eq:Pih}). 
Property (\ref{item:H-Matrix ApproximationL4-i}) is then satisfied by construction. 
In order to prove (\ref{item:H-Matrix ApproximationL4-ii}), we compute
\begin{align}
\label{eq:H-Matrix ApproximationL4-i-25}
\nonumber  \vertiii{\BE _h-\BE _{1,h}} _ {c,h,R} & = \vertiii{\boldsymbol\Pi _{B_R}\left( \BE _h +\Pi^\nabla_{\widetilde{B}_{(1+2\varepsilon) R}} \nabla \varphi_{\Bz_H} -\Bz_H- \nabla w_m\right) } _ {c,h,R}\\
\nonumber  &\le  \vertiii{\BE _h +\Pi^\nabla_{\widetilde{B}_{(1+2\varepsilon) R}} \nabla \varphi_{\Bz_H} -\Bz_H - \nabla  w_m } _ {c,h,R} \\
\nonumber & \le \vertiii{\Bz_h-\Bz }_ {c,h,R}+\vertiii{\Bz -\Bz_H  }_ {c,h,R}+\vertiii{ \Pi^\nabla_{\widetilde{B}_{(1+2\varepsilon) R}} (\nabla \varphi_{\Bz}  - \nabla \varphi_{\Bz_H}) }_ {c,h,R}\\
& \quad +\vertiii{\Pi^\nabla_{\widetilde{B}_{(1+2\varepsilon) R}} \nabla( p + \varphi_{\Bz}) - \nabla  w_m}_{c,h,R}.
\end{align}
\textbf{Step 6:} 
(stability estimates) 
The stability estimate \eqref{eq:L2-projection} for $p_h$ in the local discrete regular decomposition implies together with Lemma~\ref{Helmholtz decomposition}
\begin{align}
\label{eq:stability-10a}
\norm{\nabla p_h}_{\BLT(\widetilde B_{(1+2 \varepsilon)R})} +\|\Bz\|_{\BLT(\Omega)} +  
\|\nabla p \|_{\BLT(\Omega)} & \stackrel{\eqref{eq:L2-projection}}\lesssim \|\Bz\|_{\BLT(\Omega)} +  
\|\nabla p \|_{\BLT(\Omega)} \lesssim \|\eta \BE\|_{\BLT(\Omega)}.
\end{align}
By Lemma~\ref{local-stability-approx-divergence-free} and the Caccioppoli-type estimate of Lemma~\ref{th:Caccioppoli-divergence free part}
(replacing the pairs $(R,\varepsilon)$ there with suitably adjusted $(\widetilde R,\widetilde \varepsilon)$ as needed), we have
\begin{align}
\nonumber
\norm{\Bz_h}_{\BH(\operatorname{curl},B_{R})} + 
\norm{\Bz}_{\BH_0^1(\Omega)} &\lesssim 
\norm{\nabla \times \BE_h}_{\BLT(B_{(1+3\varepsilon)R}\cap \Omega)} + \frac{1}{\varepsilon R} \norm{\BE_h }_{\BLT(B_{(1+3\varepsilon)R}\cap \Omega)} \\
\label{eq:stability-20}
 &\stackrel{\text{Lem.~\ref{th:Caccioppoli-divergence free part}}}{\lesssim} \varepsilon^{-1} \vertiii{\BE_h}_{c,h,(1+4\varepsilon)R}. 
  \end{align}
 Finally, combining  Lemmas~\ref{Helmholtz decomposition}, \ref{lemma:varphi_z}, and \eqref{eq:L2-projection} leads to
  \begin{align}
\label{eq:stability-30}
\norm{\nabla\varphi_{\Bz}}_{\BLT(\Omega)}+\norm{\nabla\varphi_{\Bz,h}}_{\BLT(\widetilde B_{(1+2 \varepsilon)R})} &\lesssim \norm{\Bz}_{\BLT(B_{(1+3\varepsilon)R}\cap \Omega)} 
\lesssim   \|\eta \BE\|_{\BLT(\Omega)}
\end{align}
as well as
\begin{align}
\label{eq:stability-40}
\norm{\nabla(\varphi_{\Bz} -\varphi_{\Bz_H})}_{\BLT(\Omega)} &\leq \norm{\Bz - \Bz_H}_{\BLT(B_{(1+3\varepsilon)R}\cap \Omega)}. 
\end{align}
\textbf{Step 7:} (controlling $\Bz - \Bz_h$)
By Lemma~\ref{local-stability-approx-divergence-free} and (\ref{eq:stability-20}), we have 
\begin{align}\label{eq:stability-30a}
\frac{1}{R}\norm{\Bz - \Bz_h}_{\BLT(B_{R}\cap \Omega)} &\lesssim \frac{h}{R} \varepsilon^{-1} \vertiii{\BE_h}_{c,h,(1+4\varepsilon)R}. 
\end{align}
Noting   $ \nabla \times \Bz =\nabla \times (\eta \BE _h) $  together with the definition of $\vertiii{\cdot}_{c,h,R}$ and  the  estimate
	(\ref{eq:stability-30a}), we obtain 
\begin{align}
\nonumber
\vertiii{\Bz - \Bz_h}_{c,h,R}&\le \frac{h}{R}\left( \norm{\Bz_h}_{\BH(\operatorname{curl},B_{R})}+\norm{\nabla \times \Bz}_{\BLT(B_{R}\cap \Omega)}\right) + \frac{1}{R}\norm{\Bz - \Bz_h}_{\BLT(B_{R}\cap \Omega)}\\
\nonumber
&\le  \frac{h}{R}\left( \norm{\Bz_h}_{\BH(\operatorname{curl},B_{R})}+\frac{1}{\varepsilon R} \norm{\BE_h }_{\BLT(B_{(1+\varepsilon)R}\cap \Omega)}+\norm{\eta \nabla \times \BE_h }_{\BLT(B_{(1+\varepsilon)R}\cap \Omega)}\right)\\
\label{eq:H-Matrix ApproximationL4-i-10a}
&\quad + \frac{h}{R} \varepsilon^{-1} \vertiii{\BE_h}_{c,h,(1+4\varepsilon)R}. 
\end{align}
Combining this with Lemma \ref{th:Caccioppoli-divergence free part} and the stability estimate \eqref{eq:stability-20} leads to 
\begin{equation}
\label{eq:H-Matrix ApproximationL4-i-10}
\vertiii{\Bz - \Bz_h}_{c,h,R} \lesssim \frac{h}{R} \varepsilon^{-1} \vertiii{\BE_h}_{c,h,(1+4\varepsilon)R}. 
\end{equation}
\textbf{Step 8:} (controlling $\Bz - \Bz_H$ and $\nabla (\varphi_{\Bz} - \varphi_{\Bz_H})$)
For $\Bz_H = (\BI^{SZ}_H \boldsymbol \CE\Bz)|_{B_{(1+4\varepsilon)R}\cap \Omega}$,  we have by the approximation result (\ref{eq:SZ-approximation-a}),  
the assumption $H/R \leq \varepsilon/4$,  and the stability properties of $\BI^{SZ}_H$
\begin{align}\label{zh-za}
\frac{1}{R} \norm{\Bz - \Bz_H}_{\BLT(B_{(1+j\varepsilon)R}\cap \Omega)} &\lesssim \frac{H}{R} \norm{\boldsymbol \CE \Bz}_{\BH^1(B_{(1+(j+1)\varepsilon)R})}, \qquad j =0,\ldots,3, \\
\label{zh-zb}
\frac{h}{R} \norm{\Bz - \Bz_H}_{\BH^1(B_{(1+j\varepsilon)R}\cap \Omega)} &\lesssim \frac{h}{R} \norm{\boldsymbol \CE \Bz}_{\BH^1(B_{(1+(j+1)\varepsilon)R})},\qquad j =0,\ldots,3,
\end{align}
so that, using $\norm{\boldsymbol \CE \Bz}_{\BH^1(B_{(1+4\varepsilon)R })} \lesssim \norm{\Bz}_{\BH^1(\Omega)}$, we obtain 
for $j=0,\ldots,3$
\begin{equation}
\label{eq:H-Matrix ApproximationL4-i-20}
\vertiii{\Bz - \Bz_H}_{c,h,(1+j\varepsilon )R} \lesssim \left(\frac{h}{R} + \frac{H}{R}\right) \norm{\boldsymbol \CE \Bz}_{\BH^1(B_{(1+(j+1)\varepsilon)R}) }
\stackrel{\text{(\ref{eq:stability-20})}}{\lesssim} 
\left(\frac{h}{R} + \frac{H}{R}\right) \varepsilon^{-1} \vertiii{\BE_h}_{c,h,(1+4\varepsilon)R}.  
\end{equation}
By the stability properties of the operator $\Pi^\nabla_{\widetilde{B}_{(1+2\varepsilon) R}}$ given in (\ref{eq:L2-projection}) and 
(\ref{eq:stability-40}), we infer 
\begin{align}
\nonumber
\vertiii{\Pi^\nabla_{\widetilde{B}_{(1+2\varepsilon) R}} \nabla(\varphi_{\Bz} - \varphi_{\Bz_H})}_{c,h, R} & \leq 
\frac{1}{R} \norm{\nabla(\varphi_{\Bz} - \varphi_{\Bz_H})}_{\BLT(\widetilde B_{(1+2\varepsilon) R})} 
\stackrel{(\ref{eq:stability-40})}{\leq }
 \frac{1}{R} \norm{\Bz - \Bz_H}_{\BLT(B_{(1+3\varepsilon)R}\cap \Omega)} \\
\label{eq:H-Matrix ApproximationL4-i-30}
 & \stackrel{(\ref{eq:H-Matrix ApproximationL4-i-20})}{\lesssim}
\left(\frac{h}{R} + \frac{H}{R}\right) \varepsilon^{-1} \vertiii{\BE_h}_{c,h,(1+4\varepsilon)R}.
\end{align}
\textbf{Step 9:} (Estimate of $\Pi^\nabla_{\widetilde{B}_{(1+2\varepsilon) R}}\nabla(p + \varphi_{\Bz})- \nabla w_m$)
By Step~4, we have $p_h + \varphi_{\Bz,h} - w_m \in \CH_{g,h}(B_{(1+\varepsilon)R} \cap \Omega)$. 
Noting $\Pi^\nabla_{\widetilde{B}_{(1+2\varepsilon) R}}\nabla(p + \varphi_{\Bz})- \nabla w_m = \nabla (p_h + \varphi_{\Bz,h} - w_m)$ 
on $B_{(1+\varepsilon)R} \cap \Omega$, we get 
\begin{align}
\nonumber 
	\vertiii{\Pi^\nabla_{\widetilde{B}_{(1+2\varepsilon) R}} \nabla( p + \varphi_{\Bz}) - \nabla w_m}_{c,h,R}&=\frac{1}{R}\norm{ \nabla( p_h + \varphi_{\Bz,h} - w_ m)}_{\BLT(B_R \cap \Omega)}\\
\nonumber 
	&\stackrel{\text{Lem.~\ref{th:Caccioppoli-irrotational part}}}{\lesssim}
	\frac{1+\varepsilon}{\varepsilon R}\vertiii{ ( p_h + \varphi_{\Bz,h})  - w_m}_{g,h,(1+\varepsilon) R}
	\\
\nonumber 
	& \stackrel{(\ref{eq:approx-ph+phih})} {\lesssim} \frac{q^{m}  \varepsilon ^ {-2} (1+ \varepsilon)}{\varepsilon R} \norm{ \nabla( p_h + \varphi_{\Bz,h})}_{\BLT(B_{(1+2\varepsilon)R} \cap \Omega)} \\
\nonumber 
        &\stackrel{\eqref{eq:stability-10a},  \eqref{eq:stability-30} }{\lesssim} 
 \frac{q^{m}\varepsilon ^ {-2}}{\varepsilon R } \norm{ \eta \BE _h}_{\BLT (\Omega)} 
\lesssim \frac{q^{m}\varepsilon ^ {-2}}{\varepsilon R }
\norm{ \BE _h}_{\BLT (B _{(1+3\varepsilon)R} \cap \Omega)} \\
&  \lesssim q^{m}  \varepsilon ^ {-3}\vertiii{\BE_h}_{c,h,(1+3\varepsilon)R}. 
\label{eq:H-Matrix ApproximationL4-i-22} 
\end{align}
Substituting    \eqref{eq:H-Matrix ApproximationL4-i-10},   \eqref{eq:H-Matrix ApproximationL4-i-20}, \eqref{eq:H-Matrix ApproximationL4-i-30} 
and \eqref{eq:H-Matrix ApproximationL4-i-22} 
into  \eqref{eq:H-Matrix ApproximationL4-i-25}  concludes the proof of (\ref{item:H-Matrix ApproximationL4-ii}). 
\newline 
\textbf{Step 10:} By construction, the approximation $\BE_{1,h}$ of (\ref{eq:Eh1}) is from the space  
$$
\BV_{H,m}:= \{ \BPi_{B_R} (-\Pi^\nabla _{\widetilde{B}_{(1+2\varepsilon) R}} \nabla \varphi_{\Bz_H} + \Bz_H + \nabla w_m)\;\colon\;
\Bz_H \in (\BI^{\rm SZ}_H \BH^1({\mathbb R}^3)) |_{B_{(1+4\varepsilon )R} \cap \Omega}, \;w_m \in \nabla W_m\}. 
$$
By the linearity of the maps $\BPi_{B_R}$, $\Pi^\nabla_{\widetilde{B}_{(1+2\varepsilon)R}}$, and $\Bz \mapsto \varphi_{\Bz}$, the space 
$\BV_{H,m}$ is a linear space. In view of  $\operatorname{dim} W_m \lesssim (\varepsilon q)^{-3} m^4$
	from Corollary~\ref{cor:H-Matrix ApproximationL-FMP15} and 
$\dim  \BI^ {\rm  SZ}_H \boldsymbol{\CE} ( \BH ^1(\Omega))|_{B_{(1+4\varepsilon)R} \cap \Omega} \lesssim \left( \frac{(1+4\varepsilon ) R}{H}\right)  ^{3}$ 
we get (\ref{item:H-Matrix ApproximationL4-iii}). 
\qed
   \end{proof}
  \begin{lemma}[multi-step approximation] 
  	\label{H-Matrix ApproximationL5}
  	Let $ \zeta \in ( 0,1) $, $q' \in (0,1) $, $R \in (0,R_{\rm max}/2] $.
  	Then, for each $k \in {\mathbb N}$, there exists a subspace $\BV_{k} $ of   $ \mathcal{H}_{c,h}( B _{R} \cap \Omega)$ of dimension 
  	\begin{align}\label{eq:dimmultistep}
  	\dim \BV _k\leq C  _{\dim}^{\prime\prime\prime} k \left(\frac{k}{\zeta}\right)^3 \left(q'^{-3} + \ln^4\frac{k}{\zeta}\right),
  	\end{align}
  	such that for $\BE _h \in \CH _{c,h}(B _{(1+\zeta)R}\cap \Omega)$ there holds 
  	\begin{align}
  	& \min\limits_{ \widetilde{ \BE}  _{k} \in \BV_{k}}
  	\vertiii{\BE _h - \widetilde{ \BE} _{k}}_{c,h,R}
    	\leq  {q' }  ^{k}	\vertiii{\BE _h}_{h,( 1+\zeta) R}.  
  	\end{align}
  	Here, $ C  _{\dim}^{\prime\prime\prime}$ depends  only on $\kappa$, $\Omega$,  
  	and the $\gamma$-shape regularity of the quasi-uniform triangulation $\T_h$. 
   \end{lemma}
\begin{proof}
	The proof relies on iterating the approximation result of Lemma~\ref{H-Matrix ApproximationL4} 
        on boxes $B_{( 1+\varepsilon _j)R}$, where 
	$\varepsilon _j= \zeta (1 -\frac{j}{k}) $ for $j=0,\ldots,k$. We note that 
	$\zeta= \varepsilon _0 > \varepsilon _1 > \cdots >\varepsilon _k=0$.
Define 
$$
\widetilde R_j:= R (1 + \varepsilon_j), 
\qquad \widetilde \varepsilon_j:= \frac{\zeta}{4k (1 + \varepsilon_{j})} < \frac{1}{4}
$$
and note the relationship $B_{(1+4\widetilde \varepsilon_j)\widetilde R_j} = B_{\widetilde R_{j-1}} =  B_{R(1+\varepsilon_{j-1})} $
as well as $B_{\widetilde R_{k}} = B_{R}$ and $B_{\widetilde R_0} = B_{R(1+\zeta)}$.
Also note 
$$
\frac{\zeta}{8k} \leq \frac{\zeta}{4k(1+\zeta)}  \leq \widetilde \varepsilon_j \leq \frac{\zeta}{4k} , 
\qquad R \leq \widetilde R_j \leq (1+\zeta) R, \qquad j=0,\ldots,k. 
$$
Select $q \in (0,1)$. With the constant $C_{\rm app}^{\prime\prime}$ of Lemma~\ref{H-Matrix ApproximationL4} choose 
\begin{align*}
H:= \frac{q' R \zeta}{8k \max\{1,C_{\rm app}^{\prime\prime}\}}, 
\qquad m:= \left\lceil \frac{3 \ln (\zeta/(4k)) - \ln \max\{1,C_{\rm app}^{\prime\prime}\} + \ln (q'/2)}{\ln q}\right\rceil. 
\end{align*}
These constants are chosen such that 
\begin{equation}
\label{eq:multistep-10}
C_{\rm app}^{\prime\prime} \frac{H}{\widetilde\varepsilon_j \widetilde R_j} \leq \frac{1}{2} q' 
\qquad \mbox{ and } \qquad  C_{\rm app}^{\prime\prime} \widetilde \varepsilon_j^{-3} q^m \leq \frac{1}{2} q' .
\end{equation}
Moreover, the assumption $R\leq R_{\rm max}/2$ implies that $(1+4\widetilde \varepsilon_j)\widetilde R_j = R(1+\varepsilon_{j-1})\leq R_{\rm max}.$
Therefore, Lemma~\ref{H-Matrix ApproximationL4} provides a space 
$\BV^1_{H,m} \subset \mathcal{H}_{c,h}( B _{\widetilde R_1}\cap \Omega)$ and 
an approximation $\BE _ {1,h}  \in 
	\BV^1_{H,m}$
	with 
\begin{align}
\label{eq:multistep-apx}
\vertiii{\BE _h - \BE _{1,h}}_{c,h,\widetilde R_1}&\le C_{\rm app}^{\prime\prime} \left(   \frac{H}{\widetilde \varepsilon_1 \widetilde R_1}+\widetilde\varepsilon_1^{-3} q ^{m} \right) \vertiii{ \BE _h}_{c,h,\widetilde R_0} 
	 \stackrel{ (\ref{eq:multistep-10})}{\leq} 
	 q'\vertiii{\BE _h}_{c,h,\widetilde R_0},  \\
\label{eq:multistep-dim}
\nonumber\dim \BV^1_{H,m} &\lesssim\left( \frac{\widetilde R_1 }{H}\right) ^{3}+\left(\widetilde\varepsilon_1 q\right)^{-3} m ^4
\leq C  \left(\frac{k}{\zeta}\right)^3 \left( q'^{-3}+ \ln^4  (k/\zeta) \right), 
\end{align}
where the constant $C > 0$ is independent of $j \in \{0,\ldots,k\}$, $\zeta$, $k$, and $q'$. 
	Since  $\BE _h - \BE _{1,h} \in  \mathcal{H} _{c,h}( B _{\widetilde R_1} \cap \Omega)$, we may apply 
        Lemma~\ref{H-Matrix ApproximationL4} again to find a space $\BV^2_{H,m}\subset \mathcal{H}_{c,h}( B _{\widetilde R_2}\cap \Omega)$ 
        and an approximation $\BE _ {h} \in \BV^2_{H,m}$ with 
$\dim \BV^2_{H,m} \leq C (k/\zeta)^3 \left( q'^{-3} + \ln^4  (k/\zeta) \right)$ such that
	\begin{align*}
\nonumber 	\vertiii{\BE _h - \BE _{1,h}-\BE _{2,h}}_{c,h,\widetilde R_2}& \leq q ^ \prime \, \vertiii{\BE _h - \BE _ {1,h}}_{c,h,\widetilde  R_1} 
	 \leq
{q ^ \prime} ^2\, 	\vertiii{\BE _h}_{c,h,\widetilde R_0}.
	\end{align*}
	Repeating this process $k-2$ times leads to the approximation 
	$\widetilde \BE _ k=\sum_{i=1}^{k} \BE _ {i,h} $ in the space $\BV_k:= \sum_{i=1}^{k} \BV^i_{H,m}$ of dimension 
	\begin{align*}
	\dim \BV_{k} &\le  C k ({k}/{\zeta})^3 \left( q'^{-3} + \ln^4 (k/\zeta)\right), 
	\end{align*} 
which concludes the proof.
\qed
\end{proof}

%____________________________________________________________________________________________________________________________________________________________________
\section{ Proof of  main results}
\label{sec:proof}
The results of the preceding Section~\ref{sec:low-dimensional} allow us to show that the Galerkin approximation $\BE_h$ of (\ref{Galerkin discretization})
can be approximated from low-dimensional spaces in regions $B_{R_{\tau}}$ away from the support of the right-hand side $\BF$. 

   \begin{theorem}\label{pro:Hmatrix}
        Let $h_0 >0$ be given by Lemma~\ref{Lem. stability}, and let $\T_h$ be a quasi-uniform mesh with mesh size $h \leq h_0$. 
        Fix $q \in (0,1)$ and $\eta > 0$.  Set $\zeta = 1/(1+\eta)$. 
        For every cluster pair $(\tau,\sigma)$ with bounding boxes $B_{R_\tau}$ and $B_{R_\sigma}$ with 
        $\eta  \operatorname*{dist}(B_{R_\tau},B_{R_\sigma}) \ge \diam (B_{R_\tau})$  and  
   	each $k \in \mathbb{N}$,  there exists a space $\BV _{k  }   \subset \BLT(B _ {R _ \tau}\cap \Omega)$ with 
\begin{equation}
   \label{eq:pro:Hmatrix-dim}
\operatorname{dim}  \BV _{k  } \le \widetilde C  _{\dim} k (k/\zeta)^3 \left( q^{-3} + \ln^4 (k/\zeta)\right),
\end{equation}
        such that for an arbitrary right-hand side $\BF \in \BLT (\Omega)$
   	with $\operatorname*{supp} \BF \subset B_{R_\sigma} \cap \overline{\Omega} $,   
   	the corresponding Galerkin solution $\BE _h$ of \eqref{Galerkin discretization} can be  approximated 
   	from $\BV_{k  } $ such that 
   	\begin{align*}
   	\min\limits_{ \widetilde \BE_k   \in \BV_{k } } 
   \left\| \BE_h- \widetilde \BE_k \right\|_{\BLT (B_{R _ \tau} \cap \Omega)} 
   	&\le C_{\rm{box}}{q }  ^k   \norm{\BPi^{L^2} _h  \BF }_{\BLT (\Omega)}
   	 \le C_{\rm{box}} {q }  ^k \norm{  \BF }_{\BLT (B_ {R_\sigma }\cap \Omega)}. 
   	\end{align*}
   	Here, $\BPi^{L^2} _h$ is the $\BLT$-orthogonal projection onto $\BX_h(\T _h,\Omega)$ and $C_{\rm box}$, $\widetilde{C}_{\rm dim}$  
are constants depending only on $\kappa$, $\Omega$, and the shape-regularity of $\T_h$. 
   \end{theorem}
   \begin{proof}
From Lemma~\ref{Lem. stability}, we have the {\sl a priori} estimate
\begin{equation*}
\|\BE_h\|_{\BH(\operatorname{curl},\Omega)} \leq C \|\BPi^{L^2}_h \BF\|_{\BLT(\Omega)} \leq C \|\BF\|_{\BLT(\Omega)} =C \|\BF\|_{\BLT(B_{\sigma}\cap\Omega)}. 
\end{equation*}
   	{}From $\operatorname*{dist} (B_ {R_\tau}, B_ {R_\sigma})\ge \eta ^ {-1} \operatorname*{diam} B_{R_\tau}$, the choice $\zeta = 1/(1+\eta)$ implies 
   	\begin{align*}
   	\operatorname*{dist}(B_{(1+\zeta)R_ \tau},B_ {R_\sigma}) \ge 
   	\operatorname*{dist} (B_ {R_\tau},B_ {R_\sigma})- \zeta R_ \tau \sqrt{3} 
   	\ge \sqrt{3} R_\tau (\eta ^ {-1}-\zeta)=  \sqrt{3} R _\tau \frac{1}{\eta (\eta +1)} >0.
   	\end{align*} 
   	Hence, the Galerkin solution $\BE _h$ satisfies $\BE _h| _{{B_{(1+\zeta)R_\tau} } \cap \Omega}\in \mathcal{H}_{c,h} ({B_{(1+\zeta)R_\tau} }\cap \Omega) $. 
   Since $\frac{h}{R_ \tau} \lesssim 1$, it is immediate that
   \begin{align}\label{prop. proof. 1}
   \vertiii{\BE _h}_{h,( 1+\zeta) R} \lesssim  \left(  1+\frac{1}{R _ \tau}\right)\norm{\BE _h} _ {\BH  (\operatorname*{curl},\Omega)}\lesssim  \left(  1+\frac{1}{R _ \tau}\right) \norm{\BPi^{L^2}_h \BF}_ {\BLT (\Omega)}.
   \end{align}
   In the following, we employ Lemma~\ref{H-Matrix ApproximationL5}. In order to do so, 
   boxes have to have smaller side length than $R_{\rm max}/2$, which may not hold for general bounding boxes 
   $B_{R_\tau}$. However, as bounding boxes can always be chosen to satisfy $R_{\tau}<2\operatorname*{diam}(\Omega)$, 
   there exists a constant $L \in \mathbb{N}$ independent of $R_{\tau}$ such that  $R_{\tau}/L \leq  R_{\rm max}/2$ with $R_{\rm max}$ given in Def.~\ref{def:Rmax}.
Consequently,
   we can decompose a box 
   $B_{R _ \tau  }=\operatorname{int} \left( \overline{ \bigcup _{\ell =1 }^{C_L} B_{R _ {\tau _\ell}  } }\right) $
   into $C_L \in \mathbb{N}$ subboxes $\left\lbrace B_{R _ {\tau _\ell}  }  \right\rbrace _{ \ell =1} ^{C_L}$ of side length  $R _ {\tau _\ell}$
   such that 
   $R _ {\tau _\ell} \leq R_ {\rm max}/2$, where $C_L$ does only depend on $L$.
  Then, for each $B_{R _ {\tau _\ell}  },$ Lemma~\ref{H-Matrix ApproximationL5} provides a space $\BV_{k,\ell} \subset \CH_{c,h}(B_{R_{\tau_ \ell}}\cap\Omega)$, whose dimension is bounded by \eqref{eq:dimmultistep}
   such that 
 \begin{align*}
	\nonumber   \min\limits_{ \widetilde \BE_{k,\ell}   \in \BV_{k, \ell}  } 
	\left\| \BE_h- \widetilde \BE_{k, \ell}   \right\|_{\BLT (B_{R _ {\tau _\ell}  } \cap \Omega)} &\le R _ {\tau _ \ell} \min\limits_{ \widetilde \BE_{k ,\ell }  \in \BV_{k , \ell}} 
	\vertiii{\BE_h- \widetilde \BE_{k,\ell} }_{c,h,{R _ {\tau _ \ell}}} 
	\le      C {q }  ^{k} ({R _ {\tau_ \ell}+1}) \norm{\BPi^{L^2}_h \BF}_ {\BLT (\Omega)}
	\\
	&\lesssim \diam(\Omega){q }  ^{k}\norm{\BPi^{L^2}_h \BF}_ {\BLT (\Omega)}. 
	\end{align*}
Now, we define the space $\BV_k$ as a subspace of ${\BLT (B _ {R _ \tau}\cap \Omega)}$ 
by simply combining all the spaces $\BV _ {k , \ell}$ of the subboxes, i.e., 
we extend functions in $\BV _ {k , \ell}$ by zero to the larger box $B_{R_\tau}$ and write $\widehat\BV _ {k , \ell}$ for this space.
Then, we can define $\BV_k:=\sum_{\ell=1}^{C_L}\widehat\BV _ {k , \ell}$ and set $\widetilde\BE_k|_{B_{R _ {\tau _\ell}}} := \widetilde\BE_{k,\ell} \in \BV _ {k , \ell}$
for $\widetilde\BE_k \in \BV_k$.
This gives
 \begin{align*}
	\nonumber   \min\limits_{ \widetilde \BE_k   \in \BV_k } 
	\left\| \BE_h- \widetilde \BE_k  \right\|_{\BLT (B_{R _ \tau  } \cap \Omega)} &\le \sum_{\ell=1}^{C_L}\min\limits_{ \widetilde \BE_{k, \ell}   \in \BV_{k, \ell}  } 
	\left\| \BE_h- \widetilde \BE_{k, \ell}   \right\|_{\BLT (B_{R _ {\tau _\ell}  } \cap \Omega)}\\
	 & \lesssim C_L {q }  ^{k}\norm{\BPi^{L^2}_h \BF}_ {\BLT (\Omega)}.
	\end{align*}
The dimension of $\BV_k$ is bounded by  
\begin{align*}
\operatorname{dim}  \BV _{k  } \le C_L \,C  _{\dim}^{\prime\prime\prime} k \left(\frac{k}{\zeta}\right)^3 \left(q'^{-3} + \ln^4\frac{k}{\zeta}\right),
\end{align*}
which concludes the proof.
\qed
   \end{proof}
The following result allows us to transfer the approximation result of Theorem~\ref{pro:Hmatrix} to the matrix level. We recall that the system matrix $\BA$ 
is given by (\ref{eq:stiffness-matrix}).
\begin{lemma} \label{th: low rank blockwise matrix approximation of inverses} Let $h \leq h_0$ with $h_0$ given by Lemma~\ref{Lem. stability}. 
Then, there are constants $\widetilde C  _{\dim}$, $\widehat C _ {\rm  app}$ that depend only on $\kappa$, $\Omega$,  and  the 
	$\gamma$-shape regularity of the quasi-uniform triangulation $\T _h$ such that 
        for $\eta >0$, $q \in (0,1)$, $k \in \mathbb{N}$, and $\eta$-admissible cluster pairs 
	$(\tau , \sigma)$ there exist matrices $\BX _ {\tau \sigma} \in \C ^ {\tau\times r}$, $\BY _{\tau  \sigma } \in \C ^ {\sigma \times r}$ of rank $r \le \widetilde C  _{\dim} \left( {1+\eta}\right)^3{k }^4 \left( q ^{-3} + \ln^4 (k(1+\eta))\right) $ such that
	\begin{align*}
	\norm{\mathbf{A}^{-1} |_{\tau\times \sigma}-
		\mathbf{X}_{\tau\sigma}\mathbf{Y}_{\tau\sigma}^H}_2 \le \widehat C _ {\rm  app} h ^{-1}{q}^k.
	\end{align*}
\end{lemma}
   \begin{proof}
As a preliminary step, we show that we can reduce the consideration to the case $\operatorname{diam} B_{R_\tau} \leq \eta \operatorname{dist}(B_{R_\tau}, B_{R_\sigma})$. 
Indeed, as $\BA$ is symmetric also $\BA^{-1}$ is symmetric so that $\BA^{-1}|_{\tau \times \sigma} = \BA^{-1}|_{\sigma \times \tau}$ and one may approximate either
$\BA^{-1}|_{\tau\times\sigma}$ or $\BA^{-1}|_{\sigma\times\tau}$ by a low-rank matrix. In view of the definition of the admissibility condition (\ref{eq:admissibility-condition}), we may therefore assume 
$\operatorname{diam} B_{R_\tau} \leq \eta \operatorname{dist}(B_{R_\tau}, B_{R_\sigma})$. 

   	The matrices $\BX_{\tau\sigma}$ and $\BY_{\tau\sigma}$ will be constructed with the aid of Theorem~\ref{pro:Hmatrix}. In particular, let 
the constant $\widetilde C  _{\dim}$ be given from Theorem~\ref{pro:Hmatrix}. We distinguish between the cases of ``small'' blocks and ``large'' blocks. 
\newline
   	\textbf{Case 1. }  If $ \widetilde C  _{\dim} (1+\eta)^{3} {k }^4 \left( q ^ {-3} + \ln^4 (k(1+\eta))\right)  \ge \min (|\tau|, |\sigma|)$, we use the exact matrix block $\mathbf{X}_{\tau \sigma}= \mathbf{A^{-1}}|_ {\tau \times \sigma}$ and we put $\mathbf{Y}_{\tau \sigma}= \mathbf{I}|_{\sigma\times\sigma}$ with $\mathbf{I} \in \C^{N \times N}$ being the identity matrix.  \\
   	\textbf{Case 2. }  If
   	$  \widetilde C  _{\dim}  (1+\eta)^{3}{k }^4 \left( q ^ {-3} + \ln^4 (k(1+\eta))\right)  < \min (|\tau|, |\sigma|)$, let $\BV_k $ 
        be the space constructed in Theorem~\ref{pro:Hmatrix}. From $\BV_k$ we construct $\BX_{\tau\sigma}$ and $\BY_{\tau\sigma}$ in the following two steps. 
   	
   	 	\textbf{Step 1.} 
Let  functions 
$\boldsymbol{\lambda} _i \in \BLT(\Omega)$, $i=1,\ldots,N$, satisfy 
\begin{subequations}
\label{eq:lambda_i} 
\begin{alignat}{2}
\label{eq:lambda_i-a}
 \operatorname{supp} \boldsymbol{\lambda}_i & \subset \operatorname{supp} \Psi_i, && \qquad i=1,\ldots,N, \\
\label{eq:lambda_i-b}
 \langle \boldsymbol{\lambda}_i,\Psi_j\rangle_{\BLT(\Omega)}  &= \delta_{ij}, &&\qquad i,j = 1,\ldots,N, \\
\label{eq:lambda_i-c}
 \|\boldsymbol{\lambda}_i\|_{\BLT(\Omega)} &\leq C h^{-1/2}, && \qquad i=1,\ldots,N. 
\end{alignat}
\end{subequations}
Such a dual basis of $\{\Psi_i\colon i=1,\ldots,N\}$ can be constructed as (discontinuous) piecewise polynomials of degree $1$ as described 
in, e.g., \cite[Sec.~{4.8}]{brenner-scott08} for classical Lagrange elements. In fact, $\operatorname{supp} \boldsymbol{\lambda}_i$ can be taken to be a single tetrahedron in $\operatorname{supp} \Psi_i$. 
The constant $C$ depends solely on the $\gamma$-shape regularity of $\T_h$. We emphasize that our choice of 
scaling of the functions $\Psi_i$ is responsible for the factor $h^{-1/2}$. 

 For clusters $\tau'$, define the mappings
 \begin{alignat*}{2}
 \varLambda _{\tau'} & \colon \BLT(\Omega)  \rightarrow
 \mathbb{C}^{\tau'}, && \quad 
 \mathbf{v} \mapsto \left( \chi_{\tau'}(i) \langle \boldsymbol{\lambda}_i, \Bv\rangle_{\BLT(\Omega)}\right)_{i \in \mathcal{I}}, 
 \end{alignat*}
where $\chi_{\tau'}$ is the characteristic function of $\tau'$. 
For $\Bv \in \BLT(\Omega)$ and a cluster $\tau'$ with bounding box $B_{R_{\tau'}}$, 
we observe for the $\ell^2$-norm $\|\cdot\|_2$ on $\C^{\tau'}$ that 
\begin{align}
\label{eq:Lambda-tau'}
\|\varLambda_{\tau'} \Bv\|^2_{2} &  = 
\sum_{i \in \tau'} |\langle \boldsymbol{\lambda}_i, \Bv\rangle_{\BLT(\Omega)}|^2 \leq 
\sum_{i \in \tau'} \|\boldsymbol{\lambda}_i\|^2_{\BLT(\Omega)} \| \Bv\|^2_{\BLT(\operatorname{supp} \boldsymbol{\lambda}_i)} 
\stackrel{ (\ref{eq:lambda_i-c})} \lesssim h^{-1} \|\Bv\|^2_{\BLT(B_{R_{\tau'}\cap\Omega})}. 
\end{align}
We observe that, for $\BE_h \in \BX_{h,0}(\T _h,\Omega)$ expanded as $\BE_h = \sum_{i \in \mathcal{I}} \mu_i \Psi_i$, we have 
$\mu_i = (\varLambda_{\mathcal{I}} (\BE_h))_i$. In particular, we have for the coefficients $\mu_i$ with $i \in \tau'$ 
\begin{equation}
\label{eq:lambda-rep}
\mu_i = (\varLambda_{\tau'}(\BE_h) )_i 
\qquad \forall i \in \tau'. 
\end{equation}

\textbf{Step 2:} 
Let $\BV_k$ be the space given by Theorem~\ref{pro:Hmatrix} for the boxes $B_{R_\tau}$, $B_{R_\sigma}$. 
For arbitrary $\mathbf{b} \in \C^{\sigma}$, define the function 
$
\mathbf{f}_{\mathbf{b}}:= \sum_{i \in \sigma} \mathbf{b}_i \boldsymbol{\lambda}_i 
$ 
and observe: 
\begin{subequations}
\label{eq:fb}
\begin{align}
\label{eq:fb-a}
 \operatorname{supp} \mathbf{f}_{\mathbf{b}} & \stackrel{(\ref{eq:lambda_i-a})}{\subset} B_{R_\sigma}, \\
\label{eq:fb-b}
\|\mathbf{f}_{\mathbf{b}}\|_{\BLT(\Omega)} & \stackrel{(\ref{eq:Lambda-tau'})}{\lesssim} h^{-1/2} \|\mathbf{b}\|_2, \\
\label{eq:fb-c}
\langle \mathbf{f}_{\mathbf b}, \Psi_i\rangle_{\BLT(\Omega)} & \stackrel{(\ref{eq:lambda_i-b})}{=} \mathbf{b}_i, \qquad i=1,\ldots,N. 
\end{align}
\end{subequations}
Let $\BE _h \in \BX_{h,0} (\T _h,\Omega)$ be  the  Galerkin solution corresponding to the right-hand side $\mathbf{f}_{\mathbf{b}}$
and $\widetilde \BE_h \in \BV_k$ be the approximation to $\BE_h$ asserted in Theorem~\ref{pro:Hmatrix}. 
    Then, 
    \begin{align}
    \nonumber \left\| \varLambda_{\tau} \BE_h-\varLambda_{\tau}\widetilde{\BE}_k\right\|_2
    &\stackrel{ (\ref{eq:Lambda-tau'})}{ \lesssim }
h ^ {-1/2} \left\| \BE_h- \widetilde{\BE}_h\right\| _ {\BLT (B_{R_{\tau}}\cap \Omega)}  \\
    &\nonumber \stackrel{\text{Thm.~\ref{pro:Hmatrix}}}{\lesssim }  h^{-1/2}{q}^k \!      \left\|\mathbf{f}_{\mathbf{b}}\right\|_ {\BLT (\Omega)}
\stackrel{ (\ref{eq:fb-b})}{ \lesssim } h^{-1}{q}^k \norm{\mathbf{b}}_2.
    \end{align}
    We define the low-rank factor $\mathbf{X}_{\tau \sigma}$ as an orthogonal basis of the space 
     $\mathcal{V} _\tau := \{\varLambda_{\tau}(\widetilde{\BE}_k )\;\colon \; \widetilde{\BE}_k \in \BV_k\}$
    and set
     $\mathbf{Y _{\tau \sigma}}:=\mathbf{A}^{-1} |_{\tau\times \sigma}^ H\mathbf{X _{\tau \sigma}}$.
     Then, the rank of $\mathbf{X _{\tau \sigma}}$  is bounded by 
     $\dim \BV_{k} \le  \widetilde C  _{\dim}  (1+\eta)^{3}  {k }^4 \left(q ^ {-3}+ \ln^4 (k(1+\eta))\right) $.
     Since $\mathbf{X _{\tau \sigma}} \mathbf{X} _{\tau \sigma} ^H$ is the orthogonal projection from $\C ^{N}$ onto $\mathcal{V} _ \tau$, we conclude 
     that  $\Bz := \mathbf{X}_{\tau \sigma}\mathbf{X}_{\tau \sigma}^H(\varLambda_{\tau} \BE _h)$ is the $\|\cdot\|_2$-best approximation 
     of the Galerkin solution in $\mathcal{V} _\tau$, which results in 
\begin{align*}
\nonumber \norm{\varLambda_{\tau}\BE _h-\Bz}_2\lesssim
\left\| \varLambda_{\tau} \BE_h-\varLambda_{\tau}\widetilde{\BE}_h\right\|_2
\lesssim h^{-1}{q  }^k\norm{\mathbf{b}}_2.
\end{align*}
By (\ref{eq:lambda-rep}) and $\mathbf{b} \in \C^\sigma$,  we have 
\begin{align*}
\varLambda_\tau \BE_h  \stackrel{(\ref{eq:lambda-rep})}{=} (\varLambda_{\mathcal{I}} \BE_h)|_\tau = 
(\BA^{-1} \mathbf{b})|_\tau 
\stackrel{\mathbf{b} \in \C^\sigma}{= }
(\BA^{-1}|_{\tau\times\sigma})\mathbf{b}. 
\end{align*}
Since  
$\Bz = \mathbf{X}_{\tau\sigma}\mathbf{Y}_{\tau\sigma}^H\mathbf{b}$, we conclude
\begin{align*}
\norm{(\mathbf{A}^{-1} |_{\tau\times \sigma}-
	\mathbf{X}_{\tau\sigma}\mathbf{Y}_{\tau\sigma}^H)\mathbf{b}}_2 
=\norm{\varLambda_{\tau}\BE _h-\Bz}_2
\lesssim  h^{-1}{q  }^k\norm{\mathbf{b}}_2.
\end{align*}
As $\mathbf{b}$ was arbitrary, we obtain the stated norm bound. 
\qed
     \end{proof}
  \begin{proof}[Proof of Theorem \ref{th:H-Matrix approximation of inverses}]
	For each admissible cluster pair $(\tau,\sigma)$, let the matrices $\BX_{\tau\sigma}$, $\BY_{\tau\sigma}$ be given
by Lemma~\ref{th: low rank blockwise matrix approximation of inverses}. 
Define the $\mathcal{H}$-matrix approximation $\BB_{\mathcal{H}}$ by the conditions 
$$
\BB_{\mathcal{H}}|_{\tau\times\sigma} = \BX_{\tau\sigma} \BY_{\tau\sigma}^H \quad \mbox{ if $(\tau,\sigma) \in P_{\rm far}$}, 
\qquad 
\BB_{\mathcal{H}}|_{\tau\times\sigma} = \BA^{-1}|_{\tau\times \sigma}  \quad \mbox{ if $(\tau,\sigma) \in P_{\rm near}$}. 
$$
The blockwise estimate of Lemma~\ref{th: low rank blockwise matrix approximation of inverses}  for $q \in (0,1)$ and  \cite[Lemma~6]{borm2010efficient} yield 
	\begin{align*}
	\norm{\mathbf{A}^{-1}-\BB_ {\mathcal{H}}}_2 &\leq 
	C_{\rm sp} \left(\sum_{\ell=0}^{\infty}\max\{\norm{(\mathbf{A}^{-1}-\BB_ {\mathcal{H}})|_{\tau\times \sigma}  }_2 \colon (\tau,\sigma) \in P, {\rm level}(\tau) = \ell\}\right)\\
& \le 	\widehat C_{\rm  app} C_{\rm sp} \operatorname*{depth}(\mathbb{T}_{\mathcal{I}}) 
	 h^{-1}{q } ^k .
	\end{align*}
We next relate $k$ to the blockwise rank $r$. 
	For $y\ge 0$, the unique (positive) solution $k$ of $k \ln k = y$ has the form 
	\begin{equation}
	\label{eq:asymptotics-for-k}
	k = \frac{y}{\log y} (1 + o(1)) \qquad \mbox{ as $y \rightarrow \infty$}
	\end{equation}
	by, e.g., \cite[Ex.~{5.7}, Chap.~1]{olver74}. In passing, we mention that even higher order asymptotics can directly be inferred from
	the asymptotics of Lambert's $W$-function as described in \cite[p.~{25--27}]{debruijn61} or \cite[Eq.~(4.13.10)]{NIST:DLMF}. 
        The asymptotics (\ref{eq:asymptotics-for-k})
	implies that the solution $k$ of $k^4 \ln^4 k = y$ satisfies $k = y^{1/4}/\ln (y^{1/4}) (1+ o(1))$ as $y\rightarrow \infty$. 
	
	{}From Lemma~\ref{th: low rank blockwise matrix approximation of inverses}  we have  the rank bound 
 $r \le  \widetilde C  _{\dim}  (1+\eta)^{3}  {k }^4 \left( q ^ {-3} + \ln^4 (k(1+\eta))\right)  
    \le \widetilde C  _{\dim}\left( (1+\eta)q ^ {-1}\right)^{3}k^4 \ln^4 k $, 
so that, for suitable $b$, $C > 0$ independent of $r$, we get $q ^k \le  C \exp(-b  r^{1/4}/\ln r )$. Consequently,  we have
		\begin{align*}
	\left\|\mathbf{A}^{-1} -\mathbf{B}_{\mathcal{H}}
	\right\|_2 \le C_{\rm  apx} C_{\rm sp} \operatorname*{depth}(\mathbb{T}_{\mathcal{I}}) 
	h ^{-1} e^{-b \left( r^{1/4}/\ln r\right) },
	\end{align*}
	which concludes the proof.
\qed
\end{proof}
%
%_____________________________________________________________________________________________________________________________________________________________________
   \appendix
   \section{Regular decompositions}
\label{sec:appendix}
   The following lemma follows from the seminal paper \cite{costabel-mcintosh10}. 
   The notation follows \cite{costabel-mcintosh10} in that 
   $H^s_{\overline{\Omega}}({\mathbb R}^3)$, $ s \in \R$ denotes the spaces of distributions in $H ^s (\R ^3)$
   supported by $\overline{\Omega}$, and that $C^\infty_{\overline{\Omega}}({\mathbb R}^3)$
   is the space of $C^\infty({\mathbb R}^3)$-functions supported by $\overline{\Omega}$.

   We introduce the space 
   $$
   {\BH}^s_{\overline{\Omega}}(\operatorname{curl}) := 
   \{\BE \in {\BH}_{\overline{\Omega}}^s({\mathbb R}^3)\;\colon\;
   \nabla \times  \BE \in \BH_{\overline{\Omega}}^s({\mathbb R}^3)\}
   $$ 
   equipped with the norm 
   $\|\BE\|_{{\BH}_{\overline{\Omega}}^s(\operatorname{curl})}:= 
   \|\BE\|_{\BH_{\overline{\Omega}}^s({\mathbb R}^3)}+ 
   \|\nabla \times \BE\|_{\BH_{\overline{\Omega}}^s({\mathbb R}^3)}
   $.
   \begin{remark}
   	{}From \cite[p.~{301}]{costabel-mcintosh10}, for any $s \in {\mathbb R}$, 
   	the space $\BH^s_{\overline{\Omega}}({\mathbb R}^3)$
   	is naturally isomorphic to the dual space of $\BH^{-s}(\Omega)$. 
   	Hence, for $s \ge 0$, we have the alternative norm equivalence 
   	$\|{\mathbf v}\|_{\BH_{\overline{\Omega}}^s({\mathbb R}^3)} \sim  
   	\|{\mathbf v}\|_{\widetilde{\BH}^s(\Omega)} = 
   	\|{\mathbf v}^\star\|_{{\BH}^s({\mathbb R}^3)}$, where  
   	${\mathbf v}^\star$ is the zero extension of a function ${\mathbf v}$ defined on $\Omega$. 
   	\eremk
   \end{remark}
   \begin{lemma}
   	\label{LemRs}Let $\Omega$ be a bounded Lipschitz domain. There exist 
pseudodifferential operators $T_1$ and  $\mathbf{T}_2$  of order $-1$ and a 
pseudodifferential operator $\mathbf{L}$ of order $-\infty$ on $\R^3$ with the following properties: 
For each $ s \in \mathbb{R}$, they have the mapping properties 
${{T}}_{1}:\mathbf{H}_{\overline{\Omega}}^{s}%
	({\mathbb R}^3)\rightarrow {H}_{\overline{\Omega}}^{s+1}\left(  {\mathbb R}^3\right)$,
${\mathbf{T}}_{2}:\mathbf{H}_{\overline{\Omega}}^{s}%
   	({\mathbb R}^3)\rightarrow\mathbf{H}_{\overline{\Omega}}^{s+1}\left(  {\mathbb R}^3\right)$, 
and 
$\mathbf{L}:\mathbf{H}_{\overline{\Omega}}^{s}(
   	{\mathbb R}^3)  \rightarrow \mathbf{C}^{\infty}_{\overline{\Omega}}\left(  {\mathbb R}^3\right)$
and for any 
   	$\mathbf{u}\in\mathbf{H}_{\overline{\Omega}}^{s}\left(\operatorname*{curl}\right)$, there holds
   	the representation 
   	\begin{equation}
   	\mathbf{u}=\nabla T_{1}\left(  \mathbf{u}-\mathbf{T}_{2}\left(
   	\nabla \times \mathbf{u}\right)  \right)  +\mathbf{T}_{2}\left(
   	\nabla \times \mathbf{u}\right)  +\mathbf{Lu}. \label{repu}%
   	\end{equation}
   \end{lemma}
   \begin{proof}
   	In \cite[Theorem~{4.6}]{costabel-mcintosh10}, operators $T_1$,
   	$\mathbf{T}_{2}$, $\mathbf{T}_{3}$, $\mathbf{L}_{1}$, $\mathbf{L}_{2}$ with
   	the mapping properties%
   \begin{align*}
    	T_{1}&:\mathbf{H}_{\overline{\Omega}}^{s}({\mathbb R}^3)\rightarrow {H}_{\overline{\Omega}}^{s+1}\left(
    	{\mathbb R}^3\right),\\	
   	\mathbf{T}_{2}&:\mathbf{H}_{\overline{\Omega}}^{s}({\mathbb R}^3)\rightarrow\mathbf{H}_{\overline{\Omega}}^{s+1}\left(
   	{\mathbb R}^3\right),\\
   	\mathbf{T}_{3}&:H_{\overline{\Omega}}^{s}({\mathbb R}^3)\rightarrow\mathbf{H}_{\overline{\Omega}}^{s+1}\left({\mathbb R}^3\right)
   	, \\
   	\mathbf{L}_{\ell}&:\mathbf{H}_{\overline{\Omega}}^{s}\left({\mathbb R}^3\right)  \rightarrow
   	\mathbf{C}_{\overline{\Omega}}^{\infty}\left(  {\mathbb R}^3\right)  , \qquad  \ell=1,2, 
  \end{align*} 	
   	are defined, and it is shown that%
\begin{subequations}
\label{RSrelation}%
   	\begin{align}
\label{RSrelation-a}%
   	\nabla T_{1}  \mathbf{v}  +\mathbf{T}_{2}\left(
   	\nabla \times \mathbf{v}\right)   &  =\mathbf{v}-\mathbf{L}%
   	_{1}\mathbf{v},
\\
\label{RSrelation-b}%
   	\nabla \times \mathbf{T}_{2}  \mathbf{v}  +\mathbf{T}%
   	_{3}\left(  \nabla \cdot \mathbf{v}\right)   &  =\mathbf{v}%
   	-\mathbf{L}_{2}\mathbf{v}.
   	\end{align}
\end{subequations}
   	Taking $\mathbf{v}=\mathbf{u}-\mathbf{T}_{2}\left(
   	\nabla \times \mathbf{u}\right)  $ in (\ref{RSrelation-a}), we obtain 
   	%\begin{multline}
\begin{align}
   	\nabla {T}_{1}\left(  \mathbf{u}-\mathbf{T}_{2}\left(
 \nabla \times \mathbf{u}\right)  \right)  +\mathbf{T}_{2}\left(
 \nabla \times \left(  \mathbf{u}-\mathbf{T}_{2}\left(
\nabla \times \mathbf{u}\right)  \right)  \right)  
   	%\\
   	=\mathbf{u}%
   	-\mathbf{T}_{2}\left(  \nabla \times \mathbf{u}\right)  -\mathbf{L}%
   	_{1}\left(  \mathbf{u}-\mathbf{T}_{2}\left(  \nabla \times %
   	\mathbf{u}\right)  \right)  .\label{gradR1}%
   	%\end{multline}
\end{align}
   	Since $\nabla \times \mathbf{u}$ is divergence free, we obtain from (\ref{RSrelation-b}) 
        with the choice $\mathbf{v}  = \nabla \times \mathbf{u}$ 
   	\begin{align*}
   	\mathbf{T}_{2}\left( \nabla \times \left(  \mathbf{u}-\mathbf{T}%
   	_{2}\left(  \nabla \times \mathbf{u}\right)  \right)  \right)   &
   	=\mathbf{T}_{2}\left(  \nabla \times \mathbf{u}\right)  -\mathbf{T}%
   	_{2}\left(  \nabla \times \mathbf{u}-\mathbf{L}_{2}\nabla \times %
   	\mathbf{u}\right)  \\
   	&  =\mathbf{T}_{2}\left(  \mathbf{L}_{2}\left(  \nabla \times %
   	\mathbf{u}\right)  \right)  =:\mathbf{L}_{3}\mathbf{u},
   	\end{align*}
   	where, again, $\mathbf{L}_{3}$ is a smoothing operator of order $-\infty$
   	mapping into $\mathbf{C}^\infty_{\overline{\Omega}}({\mathbb R}^3)$.  
   	Inserting this into (\ref{gradR1}) leads to%
   	\[
   	\nabla T_{1}\left(  \mathbf{u}-\mathbf{T}_{2}\left(
 \nabla \times \mathbf{u}\right)  \right)  +\mathbf{T}_{2}\left(
\nabla \times \mathbf{u}\right)  =\mathbf{u}-\mathbf{L}_{1}\left(
   	\mathbf{u}-\mathbf{T}_{2}\left(  \nabla \times \mathbf{u}\right)
   	\right)  -\mathbf{L}_{3}\mathbf{u}.
   	\]
   	Choosing $\mathbf{Lu}:=\left(  \mathbf{L}_{1}\left(  \mathbf{u}%
   	-\mathbf{T}_{2}\nabla \times \mathbf{u}\right)  \right)  +\mathbf{L}%
   	_{3}\mathbf{u}$, we arrive at the representation (\ref{repu}). 
\qed
   \end{proof}

   \begin{corollary}
   	\label{cor:regular-decomp}
   	Let $\Omega \subset {\mathbb R}^3$ be a bounded Lipschitz domain. Then, 
   	for every $s \ge 0$, there is a constant $C$ (depending only on $\Omega$ and 
   	$s$) such that every $\Bu \in \BH_0^s(\operatorname{curl},\Omega)$ can 
   	be decomposed as $\Bu= \Bz + \nabla p$ with 
   	$\Bz \in\BH_{\overline{\Omega}}^{s+1}({\mathbb R}^3)$ and 
   	$p \in {H}_{\overline{\Omega}}^{s+1}({\mathbb R}^3)$ together with 
   	\begin{equation}
   	\label{eq:cor:regular-decomp}
   	\|\Bz\|_{\BH_{\overline{\Omega}}^{s+1}({\mathbb R}^3)} \leq C \|\Bu\|_{\BH_{\overline{\Omega}}^s(\operatorname{curl})}, 
   	\qquad 
   	\|\nabla p\|_{\BH_{\overline{\Omega}}^{s}({\mathbb R}^3)} \leq C \|\Bu\|_{{\BH}_{\overline{\Omega}}^s({\mathbb R}^3)}. 
   	\end{equation}
   \end{corollary}
   \begin{proof} 
   	{}From Lemma~\ref{LemRs} we can write 
   	${\mathbf u}  = \Bz + \nabla p$ with 
   	\begin{align*}
   	\Bz &:= 
   	\mathbf{T}_{2}\left(
\nabla \times \mathbf{u}\right)  +\mathbf{Lu}, 
   	& 
   	p &:= T_{1}\left(  \mathbf{u}-\mathbf{T}_{2}\left(
\nabla \times \mathbf{u}\right)  \right)  . 
   	\end{align*}
   	The stability estimate for $\Bz$ follows from the mapping properties of 
   	the operators ${\mathbf T}_2$ and ${\mathbf L}$.  
   	The mapping properties of 
   	${\mathbf T}_1$ yield 
   	\begin{align*}
   	\|\nabla p \|_{\BH^s_{\overline{\Omega}}({\mathbb R}^3)} & \lesssim 
   	\|{\mathbf u}  - {\mathbf T}_2(\nabla \times  {\mathbf u})\|_{{\BH^s_{\overline{\Omega}}({\mathbb R}^3)}}
   	\lesssim 
   	\|{\mathbf u} \|_{{\BH^s_{\overline{\Omega}}({\mathbb R}^3)}}
   	+ \|\nabla \times {\mathbf u} \|_{{\BH^{s-1}_{\overline{\Omega}}({\mathbb R}^3)}} 
%\\ & 
\lesssim \|{\mathbf u} \|_{{\BH^s_{\overline{\Omega}}({\mathbb R}^3)}}, 
   	\end{align*}
   	where the last step follows from the mapping property 
   	$\nabla \times :\BH^s_{\overline{\Omega}}({\mathbb R}^3)  
   	\rightarrow \BH^{s-1}_{\overline{\Omega}}({\mathbb R}^3)$. 
\qed
   \end{proof}

 \section*{Funding}
Financial support by the Austrian Science Fund (FWF) through the
research program ``Taming complexity in partial differential systems'' (grant SFB F65) for JMM, through grant P 28367-N35 for JMM and MP and by the Deutsche Forschungsgemeinschaft (DFG) under Germany’s
Excellence Strategy within the Cluster of Excellence PhoenixD (EXC 2122, 
	Project ID 390833453) for MP is gratefully acknowledged.

\bibliographystyle{amsalpha}
\bibliography{FEM}   % name your BibTeX data base

\end{document}